\documentclass[reqno]{amsart}
\usepackage{amssymb}
\usepackage{graphicx}

\usepackage[usenames, dvipsnames]{color}
\usepackage{verbatim}
\usepackage{mathrsfs}
\usepackage{bm}
\usepackage{cite}
\usepackage{bbm}

\numberwithin{equation}{section}

\newtheorem{theorem}{Theorem}[section]
\newtheorem{corollary}[theorem]{Corollary}
\newtheorem{lemma}[theorem]{Lemma}
\newtheorem{prop}[theorem]{Proposition}

\theoremstyle{definition}
\newtheorem{remark}[theorem]{Remark}

\theoremstyle{definition}

\theoremstyle{definition}

\makeatletter
\def\dashint{\operatorname%
{\,\,\text{\bf-}\kern-.98em\DOTSI\intop\ilimits@\!\!}}
\makeatother

\def\\det{\text{\det}}

\def\Xint#1{\mathchoice
 {\XXint\displaystyle\textstyle{#1}}%
 {\XXint\textstyle\scriptstyle{#1}}%
 {\XXint\scriptstyle\scriptscriptstyle{#1}}%
 {\XXint\scriptscriptstyle\scriptscriptstyle{#1}}%
 \!\int}
\def\XXint#1#2#3{{\setbox0=\hbox{$#1{#2#3}{\int}$}
  \vcenter{\hbox{$#2#3$}}\kern-.5\wd0}}

\def\dashint{\Xint-}

\def\.5{\frac{1}{2}}

\newcommand{\RN}[1]{%
  \textup{\uppercase\expandafter{\romannumeral#1}}%
}

\newcounter{marnote}

\begin{document}

\title[Regularity for quasilinear parabolic equations]{Local regularity for solutions to singular and degenerate quasilinear parabolic equations}

\author[Y. Li]{Yan Li}
\address[Y. Li] {School of Mathematical Sciences, Beijing Normal University, Beijing 100875, China. }
\email{yanli25@mail.bnu.edu.cn}

\author[Z.W. Zhao]{Zhiwen Zhao}

\address[Z.W. Zhao]{School of Mathematics and Physics, University of Science and Technology Beijing, Beijing 100083, China.}
\email{zwzhao365@163.com;\,zwzhao365@gmail.com}


\date{\today} 


\maketitle
\begin{abstract}

The main purpose of this paper is to prove local H\"{o}lder continuity for weak solutions to quasilinear parabolic equations of $p$-Laplace type under finite
integrability assumptions on the forcing, covering the supercritical singular range and the entire degenerate regime.\ The key ingredient is an $L^\infty$-valued temporal H\"{o}lder estimate of the solution after any fixed positive waiting time, which is achieved by coupling a time-rescaling argument with De Giorgi iteration.\ We also provide counterexamples involving forcing terms with weaker spatial integrability to show that even bounded continuous solutions with compatible data may possess no positive temporal H\"{o}lder exponent.\ In particular, these examples establish the sharpness of the spatial integrability threshold in the degenerate case.\ Finally, a separate counterexample is constructed to demonstrate that the temporal Lipschitz endpoint can fail for bounded continuous forcing when it depends on time.

\end{abstract}

\maketitle

%

\section{Introduction}
For $n\geq2$, let $\Omega$ be a bounded domain in $\mathbb{R}^{n}$ with smooth boundary.\ Given $T>0$, define $\Omega_T=\Omega\times(-T,0]$, whose parabolic boundary is denoted by $\partial_{pa}\Omega_T$. Write
\begin{align}\label{CRI01}
\theta_{\mathrm{c}}=\frac{2n}{n+2},\quad n\geq2.	
\end{align}	
For $p\in(\theta_{\mathrm{c}},2)\cup(2,\infty)$, in this paper we consider the following boundary value problem:
\begin{equation}\label{PO001}
	\begin{cases}
		\partial_{t}u-\mathrm{div}\bigl(\mathbf{a}(x,\nabla u)\bigr)=b(x,t,u),&\text{in }\Omega_{T},\\
		u=g(x)\;\,\text{or}\;\,
		\mathbf{a}(x,\nabla u)\cdot\nu=0,&\text{on }\partial \Omega\times(-T,0],\\
		u(x,-T)=\phi(x),&\text{in }\Omega,
	\end{cases}
\end{equation}
where $g\in L^{\infty}(\partial\Omega)$, $\phi\in L^{\infty}(\Omega)$, $\nu$ stands for the outward unit normal to $\partial\Omega$. The two boundary conditions are alternatives, not simultaneous requirements. The vector field $\mathbf{a}:\Omega\times\mathbb{R}^{n}\to\mathbb{R}^{n}$ is measurable in $x$ and continuous in $\xi$.\ For almost every $x\in\Omega$, all $\xi,\eta\in\mathbb{R}^{n}$ and $\kappa>0$, assume that
\begin{align}\label{E01}
	\mathbf{a}(x,\xi)\cdot\xi\geq L_{1}|\xi|^{p},\quad |\mathbf{a}(x,\xi)|\le L_{2}|\xi|^{p-1},\quad \mathbf{a}(x,\kappa\xi)=\kappa^{p-1}\mathbf{a}(x,\xi),
\end{align}
and
\begin{align}\label{E32}
	\bigl(\mathbf{a}(x,\xi)-\mathbf{a}(x,\eta)\bigr)\cdot(\xi-\eta)\geq L_{3}
	\begin{cases}
		\frac{|\xi-\eta|^2}{(1+|\xi|+|\eta|)^{2-p}},&\text{for }\theta_{\mathrm{c}}<p<2,\\
		|\xi-\eta|^{p},&\text{for }p>2,
	\end{cases}	
\end{align}	
where $L_{i}>0$, $i=1,2,3$.\ From Chapter 12 in \cite{L2019}, we see that the classical flux $|\xi|^{p-2}\xi$ obeys these conditions.\ The lower-order source term $b:\Omega_{T}\times\mathbb{R}\to\mathbb{R}$ is measurable and satisfies
\begin{equation}\label{E02}
	|b(x,t,u)|\leq \psi(x,t),\quad\psi\geq0,\quad \psi\in L^{\frac{pl_{0}}{p-1}}(\Omega_{T})\cap L^{\frac{p}{p-1}}((-T,0);L^{q}(\Omega)).
\end{equation}
Here $l_{0}>\frac{n+p}{p}$, and 
\begin{equation}\label{eq:q-critical}
q\in(q_{\mathrm{c}},\infty),\quad	q_{\mathrm{c}}=
	\begin{cases}
	\frac{2np}{p(n+2)-2n},&\text{for }\theta_{\mathrm{c}}<p<2,\\
		n,&\text{for }p>2.
	\end{cases}
\end{equation}

 We say that a function $u\in C([-T,0];L^{2}(\Omega))\cap L^{p}((-T,0);W^{1,p}(\Omega))$ is called a weak solution to problem \eqref{PO001} provided that for every
$-T\leq t_{1}<t_{2}\leq0$,
\begin{align*}
	&\int_{\Omega}u\varphi\,dx\Big|_{t_{1}}^{t_{2}}+\int_{t_{1}}^{t_{2}}\int_{\Omega}\left(-u\partial_{t}\varphi
	+\mathbf{a}(x,\nabla u)\cdot\nabla\varphi\right)\,dxdt=\int_{t_{1}}^{t_{2}}\int_{\Omega}b(x,t,u)\varphi\,dxdt,
\end{align*}
for any $\varphi\in W^{1,2}((-T,0);L^{2}(\Omega))\cap L^{p}((-T,0);W^{1,p}_{0}(\Omega))$ in the Dirichlet case and $\varphi\in W^{1,2}((-T,0);L^{2}(\Omega))\cap L^{p}((-T,0);W^{1,p}(\Omega))$ in the Neumann case. For problem \eqref{PO001}, the assumed energy regularity of the solution, together with the Dirichlet trace condition, already entails $g\in W^{1-\frac{1}{p},p}(\partial\Omega)$, where $W^{1-\frac{1}{p},p}(\partial\Omega)$ denotes the fractional Sobolev trace space of $W^{1,p}(\Omega)$. In fact, recall the classical trace theorem: the operator
\begin{align}\label{TRACE}
	\mathrm{Tr}:W^{1,p}(\Omega)\longrightarrow W^{1-\frac{1}{p},p}(\partial \Omega)
\end{align}
is bounded and surjective and has kernel $W^{1,p}_{0}(\Omega)$, see p.\ 315 in \cite{B2011}.\ Making use of \eqref{TRACE}, we deduce
\begin{align*}
	\|g\|_{W^{1-\frac{1}{p},p}(\partial\Omega)}^{p}=&\frac{1}{T}\int_{-T}^{0}\|\mathrm{Tr}u(\cdot,t)\|_{W^{1-\frac{1}{p},p}(\partial\Omega)}^{p}dt\\
	\leq &\frac{C(\Omega,p)}{T}\int_{-T}^{0}
	\|u(\cdot,t)\|_{W^{1,p}(\Omega)}^{p}dt<\infty.
\end{align*}

Choose
\begin{equation}\label{eq:theta}
	\frac{q_{\mathrm{c}}}{q}<\vartheta<1,
\end{equation}
where $q$ and $q_{\mathrm{c}}$ are given by \eqref{eq:q-critical}.\ For $a\in\mathbb{R}$, let $a_+=\max\{a,0\}$.\ For convenience of notation, define
\begin{equation}\label{eq:unified-exponents}
	\begin{gathered}
		\beta=\frac{1-\vartheta}{p+\vartheta(p-2)_{+}},\quad 
		\alpha=\frac{\beta}{1+\beta}
		=\frac{1-\vartheta}{p+1+\vartheta\bigl((p-2)_+-1\bigr)}.
	\end{gathered}
\end{equation}
Write
\begin{align}\label{EXP01}
\varepsilon_{0}=
\begin{cases}
\min\{p\alpha,\varepsilon\},&\text{for }\theta_{\mathrm{c}}<p<2,\vspace{0.2ex}\\
\min\big\{\frac{p\alpha}{1+\alpha(p-2)},\varepsilon\big\},&\text{for }p>2,
\end{cases}
\end{align}
where $\theta_{\mathrm{c}}$ is given by \eqref{CRI01}, and
\begin{align*}
\quad\varepsilon=
	\begin{cases}
		\frac{p(p(l_{0}-1)-n)}{2pl_{0}+n(p-2)},&\mathrm{for}\;\theta_{\mathrm{c}}<p<2,\vspace{0.3ex}\\
		\frac{p(l_{0}-1)-n}{p(l_{0}-1)+2},&\mathrm{for}\;p>2.
	\end{cases}
\end{align*}

Throughout the paper, we use $C$ to represent a generic positive constant depending solely on the prescribed data, including $n$, $p$, $q$, $l_{0},$ $\vartheta$, $\|\phi\|_{L^{\infty}(\Omega)}$, $\|g\|_{L^{\infty}(\partial\Omega)}$, $\|g\|_{W^{1-\frac{1}{p},p}(\partial\Omega)}$, $\|\psi\|_{L^{\frac{pl_{0}}{p-1}}(\Omega_T)}$, and $L_i$, $i=1,2,3$, whose value may vary from line to line.

Our main results are now listed as follows.
\begin{theorem}\label{ZWTHM90}
Fix $p\in(\theta_{\mathrm{c}},2)\cup(2,\infty)$ and $n\geq2$.\ Let $u$ be a weak solution of \eqref{PO001} with $\Omega=B_1$ and $T=1$.\ Then one can find an exponent $\hat{\alpha}\in(0,\varepsilon_{0}]$, depending only on the prescribed data, such that for all $(x,t),(y,s)\in B_{\frac{1}{16}}\times(-\frac{1}{2},0]$,
\begin{align}\label{QNAW001ZW}
|u(x,t)-u(y,s)|\leq C\big(|x-y|+|t-s|^{\frac{1}{p}}\big)^{\hat{\alpha}}.
\end{align}

\end{theorem}

\begin{remark}
Let $\Subset$ represent compact containment.\ For a general space-time cylinder $\Omega_T$, choose $\Omega_{0}\Subset\Omega$.\ Set $R_{0}=\min\{1,\mathrm{dist}(\Omega_{0},\partial\Omega),T^{\frac{1}{p}}\}$.\ By adapting the proof of Theorem \ref{ZWTHM90} with minor modifications, we obtain the H\"older estimate in \eqref{QNAW001ZW} with $B_{1/16}\times(-2^{-1},0]$ replaced by $\Omega_{0}\times(-T/2,0]$.\ Here the constant $C$ may depend on $R_{0}$, whereas the exponent $\alpha$ is independent of $R_{0}$.

\end{remark}

The key to proving Theorem \ref{ZWTHM90} lies in establishing the following temporal H\"older regularity with values in $L^\infty$.
\begin{theorem}\label{thm10}
Set $n\geq2$.\ Given $p\in\big(\theta_{\mathrm{c}},2\big)\cup(2,\infty)$, let $u$ be the solution of \eqref{PO001} in $B_{1}\times(-1,0]$. For any $0<\delta\ll 1$, we derive that for every $-1+\delta\leq s<t\leq 0$,
\begin{align*}
			\|u(\cdot,t)-u(\cdot,s)\|_{L^{\infty}(B_{1})}\leq C\delta^{-\beta}|t-s|^{\alpha},
\end{align*}
where $\alpha$ and $\beta$ are defined by \eqref{eq:unified-exponents}.
\end{theorem}

For $p>1$, denote $p'=\frac{p}{p-1}$. To construct a counterexample showing that forcing in $L^{p'}_tL^n_x$ does not guarantee temporal H\"older continuity of any positive order, we consider the following initial--boundary value problem
\begin{equation}\label{eq:model}
	\begin{cases}
		\partial_tu-\mathrm{div}\big(|\nabla u|^{p-2}\nabla u\big)=F(x,t),
		&\text{in }B_1\times(-1,0),\\
		u=1\;\,\text{or}\;\, |\nabla u|^{p-2}\nabla u\cdot\nu=0,
		&\text{on }\partial B_1\times(-1,0],\\
		u(x,-1)=1,&\text{in }B_1.
	\end{cases}
\end{equation}
Here $\nu$ is the outward unit normal. For a nondegenerate time interval $I\subset(-1,0]$ and $0<\bar{\alpha}\leq1$, define the extended seminorm
\begin{equation*}
	[u]_{C^{0,\bar{\alpha}}(I;L^\infty(B_1))}
	=
	\sup_{\substack{s,t\in I\\s<t}}
	\frac{
		\|u(\cdot,t)-u(\cdot,s)\|_{L^\infty(B_1)}
	}{(t-s)^{\bar{\alpha}}}
	\in[0,\infty].
\end{equation*}
Its finiteness characterizes H\"older continuity in time with values in $L^\infty$.\ The case of $\bar{\alpha}=1$ corresponds to Lipschitz continuity.

\begin{prop}\label{prop:critical}
Fix $n\geq2$ and $p>1$.\ Each problem in \eqref{eq:model} admits a source $F\in L^{p'}((-1,0);L^n(B_1))$ and an energy weak solution $u\in C(\overline{B}_{1}\times[-1,0])\cap L^p((-1,0);W^{1,p}(B_1))$ with $1\leq u\leq2$ such that
	\begin{equation}\label{eq:failure-claim}
		[u]_{C^{0,\bar{\alpha}}((-1/2,-1/2+\rho);L^\infty(B_1))}=\infty,
		\quad\text{for every }0<\bar{\alpha}\leq1,
		\; 0<\rho\leq1/2.
	\end{equation}
	
\end{prop}

\begin{remark}
	For $p>2$, Proposition \ref{prop:critical} establishes the sharpness of the condition $q>n$ for the temporal H\"older regularity by exhibiting failure when $q\leq n$.\ The counterexample extends to $1<p<2$, but in the range $2n/(n+2)<p<2$, the sufficient threshold $q_{\mathrm c}=2np/(p(n+2)-2n)$ is strictly larger than $n$.\ Thus, the intermediate range $n<q\leq q_{\mathrm c}$ remains unresolved.
	
\end{remark}

We next construct a counterexample to show that continuous bounded time-dependent forcing terms can cause the failure of Lipschitz continuity in time with values in $L^{\infty}$. For this purpose, we consider
\begin{equation}\label{eq:counterexample-model}
	\partial_tu-\mathrm{div}\left(|\nabla u|^{p-2}\nabla u\right)=F(x,t),\quad\text{in }B_1\times(-1,0].
\end{equation}
For the Dirichlet problem, prescribe
\begin{equation}\label{eq:counterexample-D-data}
	\begin{cases}
		u(x,t)=2+x_1,
		&\text{on }\partial B_1\times(-1,0],\\
		u(x,-1)=2+x_1,
		&\text{in }B_1.
	\end{cases}
\end{equation}
For the homogeneous Neumann problem, prescribe
\begin{equation}\label{eq:counterexample-N-data}
	\begin{cases}
		|\nabla u|^{p-2}\nabla u\cdot\nu=0,
		&\text{on }\partial B_1\times(-1,0],\\
		u(x,-1)=2+\big(1-|x|^{p'}\big)^{p'},
		&\text{in }B_1.
	\end{cases}
\end{equation}
Both initial data satisfy the respective initial--boundary compatibility conditions.
\begin{prop}\label{prop:main}
Set $n\geq2$ and $p>1$.\ For each of the two sets of initial and boundary data \eqref{eq:counterexample-D-data} and \eqref{eq:counterexample-N-data}, there exist a time-dependent source $F\in C(\overline{B}_{1}\times[-1,0])$ and a positive bounded energy weak solution $u$ of \eqref{eq:counterexample-model} with bounded spatial gradient such that the solution belongs to $C^{\bar{\alpha}}([-1,0];L^\infty(B_1))$ for any $0<\bar{\alpha}<1$, but its $L^\infty$-valued temporal Lipschitz seminorm is infinite on every right neighborhood of $-\frac{1}{2}$. 
\end{prop}
\begin{remark}
This counterexample actually demonstrates the sharpness of the time-independence assumption on the forcing in Theorem 4.1 of \cite{HLZ202601} and Theorem 3.1 of \cite{HLZ202602}.\ While the proofs therein require time-independent forcing to implement a time-rescaling technique combined with $L^\infty$ comparison estimates, our counterexample further shows that this assumption cannot, in general, be relaxed.
\end{remark}

%


The rest of the paper is arranged as follows. Section \ref{SEC002} contains some preliminary results. Section \ref{SEC003} focuses on establishing the
temporal H\"older continuity with values in $L^\infty$.\ The counterexamples in Propositions \ref{prop:critical} and \ref{prop:main} are constructed in Section \ref{SEC005}.\ In Sections \ref{SEC006} and \ref{SEC007}, we establish the three key ingredients needed to implement the parabolic De Giorgi truncation method and thus complete the proof of Theorem \ref{ZWTHM90} for the supercritical singular and degenerate cases, respectively.\ We now close the introduction by reviewing the relevant literature and explaining how the present results relate to and differ from previous work.

Both theoretical developments and practical applications have contributed to the extensive study of the parabolic $p$-Laplace equation, which serves as a prototype for the class of quasilinear parabolic equations considered above.\ Models related to porous-media transport, image processing and game theory provide important motivations, see e.g. \cite{DLK2013,ACM2004,BV2004,PS2008}. The qualitative behavior of the parabolic $p$-Laplace equation varies significantly with the exponent $p$. The classical heat equation arises when $p=2$, whereas the nonlinear cases on either side of this threshold exhibit markedly different diffusion mechanisms.\ In the case when $1<p<2$, diffusion is fast enough to permit finite-time extinction under appropriate boundary conditions. For $p>2$, diffusion becomes slow and relaxation toward equilibrium typically follows an algebraic decay law, in contrast to the exponential decay associated with the linear case.

Unlike the classical heat equation, whose solutions are smooth in the interior, the nonlinear parabolic $p$-Laplace equation lacks uniform ellipticity. Indeed, the diffusion coefficient $|\nabla u|^{p-2}$ becomes unbounded as $|\nabla u|\to0$ when $1<p<2$, whereas it approaches zero in the same limit when $p>2$.\ These two mechanisms, known as singularity and degeneracy, respectively, create substantial difficulties in establishing higher regularity.\ Spatial $C^{1,\tilde{\alpha}}$ regularity with $0<\tilde{\alpha}<1$ is a classical regularity conclusion for such nonlinear equations.\ Understanding these phenomena has driven extensive developments in the regularity theory of quasilinear elliptic and parabolic equations, see e.g. \cite{BDGLS2026,CM2011,CM2014,L2019,D1986,CD1988,DZ2022,DZ2024,D1983,D1993,DM2010,DM201002,DM2011,U2008,E1982,DK1992,L1994,DGV2012,DGV2008,DGV200802,AL1983,DF198501,DF198502,DF1984,C1991,BSS2022,KL2000,L2006}.\ This line of research also extends to broader classes of doubly nonlinear parabolic equations, see e.g. \cite{I1989,I199401,I199402,I1995,BDMS2018,BDL2021,BDGLS202301,KSU2012,PV1993,VV2022,BFV2018,V2007,DK2007,DKV1991,JX2019,JRX2023,K1988,BDGLS202302}.

The extension of the De Giorgi truncation method created in \cite{D1957} to nonlinear parabolic equations of $p$-Laplace type encounters a fundamental difficulty arising from the different homogeneities of the temporal and spatial energy terms.\ More precisely, the time derivative produces a quadratic energy contribution, whereas the diffusion operator gives rise to a gradient term of order $p$. These two structures are naturally compatible when $p=2$, allowing the elliptic De Giorgi iteration in \cite{D1957} to be adapted to the heat equation through the Sobolev embedding theorem. The De Giorgi truncation method is built upon a priori integral estimates controlling level sets near the local extrema of weak solutions. Once obtained, these estimates reduce the regularity
problem to a purely analytical one, irrespective of the differential equation or variational problem from which they originate. The parabolic argument proceeds through three interconnected steps: forward propagation of spatial level-set information, decay of space-time level-set measures, and measure-to-point oscillation improvement.\ For $p\neq2$, the mismatch of homogeneities between the temporal energy and spatial dissipation terms leads to distinct difficulties in the two diffusion regimes: the main obstacle lies in establishing the measure-decay estimate for $1<p<2$, while the propagation estimate poses the principal challenge for $p>2$.

A fundamental advance in the regularity theory for  degenerate quasilinear parabolic equations was made by DiBenedetto \cite{D1986,D1993}, who developed the intrinsic-scaling technique to restore the appropriate balance between the competing terms in the energy estimates. The essential idea is to adapt the geometry of space-time cylinders to the local oscillation of the solution rather than relying on fixed scales. Subsequently, DiBenedetto, Gianazza, and Vespri
\cite{DGV2008,DGV2012} refined this approach by means of an exponential change of variables, which enables the propagation of measure information
over sufficiently long intrinsic time intervals, thereby overcoming the difficulty associated with temporal expansion.\ In particular, this method led to the establishment of H\"{o}lder regularity for weak solutions of general quasilinear degenerate parabolic equations with lower-order terms.

A new route to the regularity theory of weak solutions to the classical parabolic $p$-Laplace equation has recently been developed by Hao, Li and Zhao \cite{HLZ202601,HLZ202602}, who introduced a time-rescaling argument coupled with $L^{\infty}$ comparison estimates to derive Lipschitz continuity in time away from the initial time.\ With this substantial gain in temporal regularity, the space-time level-set measure-decay problem can be reduced to an elliptic-type spatial distribution decay estimate on almost every time slice, while the forward-in-time propagation of level-set measures can be established directly without relying on the energy inequality underlying the earlier approach.\ As a result, they refined the method of intrinsic scaling and established H\"{o}lder continuity of weak solutions
for $1<p<2$, together with an alternative proof for $p>2$.

In this paper, we consider the quasilinear parabolic problem in \eqref{PO001} involving a forcing term satisfying finite integrability conditions.\ The crucial step in adapting the strategy proposed by Hao, Li and Zhao \cite{HLZ202601,HLZ202602} to the current structural assumptions remains to enhance temporal regularity. In order to overcome the difficulty arising from the lower-order source term, we further develop a time-rescaling technique coupled with De Giorgi iteration. Specifically, for the purpose of accommodating the lower-order term with only finite mixed-norm integrability, we retain the time-rescaling step of the earlier temporal Lipschitz argument but replace its $L^\infty$ comparison estimates by De Giorgi iteration. The rescaled forcing discrepancy need not be bounded, so the previous supremum-norm comparison argument no longer closes. Instead, quantitative monotonicity yields truncated energy estimates for the rescaled difference, while a support-sensitive Sobolev inequality converts the mixed-norm control of the forcing into a superlinear level-set estimate.\ Iteration followed by optimization of the comparison time scale then gives temporal H\"older continuity with values in $L^\infty$ after any given positive elapsed time. In the singular case, this approach requires the supercritical range, since the underlying energy--Sobolev step loses the strict support-measure gain at and below the critical threshold.\ This restriction reflects a limitation of the present method.

\section{Preliminary}\label{SEC002}
We first fix the notation to be used throughout the paper. For any $x_{0}\in\mathbb{R}^{n}$ and $\rho>0$, write $B_{\rho}(x_{0})$ for the ball of radius $\rho$ centered at $x_{0}$. Given $t_{0}\leq0$ and $\tau>0$, we define the backward parabolic cylinder by
\begin{align*}
[(x_{0},t_{0})+Q(\rho,\tau)]=B_{\rho}(x_{0})\times(t_{0}-\tau,t_{0}].
\end{align*}
To simplify the notation, we write $B_{\rho}=B_{\rho}(0)$ and $Q(\rho,\tau)=[(0,0)+Q(\rho,\tau)]$. 

We recall two fundamental tools for the De Giorgi truncation method:\ the isoperimetric inequality and the Sobolev embedding theorem.
\begin{lemma}[see Lemma 2.2 in Chapter 2 of \cite{DGV2012}]
Let $n\geq2$ and $1<p<\infty$. Then we derive that for every $B_{R}(x_{0})\subset \mathbb{R}^{n}$, $l>k$ and $u\in W^{1,1}(B_{R}(x_{0}))$,
\begin{align}\label{pro001}
&(l-k)|\{u\geq l\}\cap B_{R}(x_{0})|\leq \frac{C(n,p)R^{n+1}}{|\{u\leq k\}\cap B_{R}(x_{0})|}\int_{\{k<u<l\}\cap B_{R}(x_{0})}|\nabla u|dx.
\end{align}
\end{lemma}
A direct application of Proposition 3.1 in Chapter I of \cite{D1993} and Proposition 4.1 in Chapter 2 of \cite{DGV2012} implies the following Sobolev inequalities.
\begin{lemma}\label{lem001}
Set $n\geq2$. Let $B_{R}(x_{0})\subset \mathbb{R}^{n}$ and $-1\leq t_{1}<t_{2}\leq0$. Then we have

$(i)$ if $1<p<2$, for any  $u\in L^{2}(B_{R}(x_{0}))\cap W_{0}^{1,p}(B_{R}(x_{0}))$,
\begin{align*}
\int_{B_{R}(x_{0})}|u|^{p\tilde{\chi}}dx\leq C\int_{B_{R}(x_{0})}|\nabla u|^{p}dx\bigg(\int_{B_{R}(x_{0})}|u|^{2}dx\bigg)^{\chi-1};
\end{align*}

$(ii)$ if $p>2$, for any $u\in C((t_{1},t_{2});L^{p}(B_{R}(x_{0})))\cap L^{p}((t_{1},t_{2});W_{0}^{1,p}(B_{R}(x_{0})))$,
\begin{align*}
	\|u\|_{L^{p\chi}(B_{R}(x_{0})\times(t_{1},t_{2}))}\leq C\bigg(\sup\limits_{t\in(t_{1},t_{2})}\int_{B_{R}(x_{0})}|u|^{p}dx+\int^{t_{2}}_{t_{1}}\int_{B_{R}(x_{0})}|\nabla u|^{p}dxdt\bigg)^{\frac{1}{p}},
\end{align*}
where $C=C(n,p)$, $\chi$ and $\tilde{\chi}$ are given by
\begin{align}\label{chi}
\chi=\frac{n+p}{n},\quad \tilde{\chi}=\frac{n+2}{n}.
\end{align}

\end{lemma}

For $k\in\mathbb{R}$ and $u\in C([-1,0];L^{2}(B_{1}))\cap L^{p}((-1,0);W^{1,p}(B_{1}))$, let
\begin{align*}
	(u-k)_{+}=\max\{u-k,0\},\quad(u-k)_{-}=\max\{k-u,0\}.
\end{align*}
A consequence of the proof of Lemma 3.1 in \cite{HLZ202601} with minor modifications gives the following energy inequality.
\begin{lemma}\label{lem003}
Suppose that $u$ satisfies \eqref{PO001} in $B_{1}\times(-1,0]$. Fix $[(x_{0},t_{0})+Q(\rho,\tau)]\subset B_{1}\times(-1,0]$ and write $v_{\pm}=(u-k)_{\pm}$ with $k\in\mathbb{R}$. Let $\xi\in C^{\infty}([(x_{0},t_{0})+Q(\rho,\tau)])$ be any cutoff function satisfying $0\leq\xi\leq1$ and $\xi=0$ on
	$\partial B_{\rho}(x_{0})\times(t_{0}-\tau,t_{0})$. Then for every $s\in(t_{0}-\tau,t_{0})$, we have
	\begin{align}\label{ENE01}
		&\int_{B_{\rho}(x_{0})}v_{\pm}^{2}\xi^{p}(x,s)dx+\frac{L_{1}}{6}\int_{B_{\rho}(x_{0})\times(t_{0}-\tau,s)}|\nabla(v_{\pm}\xi)|^{p}dxdt\notag\\
		&\leq\int_{B_{\rho}(x_{0})}v_{\pm}^{2}\xi^{p}(x,t_{0}-\tau)dx+C\int_{[(x_{0},t_{0})+Q(\rho,\tau)]}\big(v_{\pm}^{2}|\partial_{t}\xi|+v_{\pm}^{p}|\nabla\xi|^{p}\big)dxdt\notag\\
		&\quad+C\|\psi\|^{\frac{p}{p-1}}_{L^{\frac{pl_{0}}{p-1}}(B_{1}\times(-1,0))}|[(x_{0},t_{0})+Q(\rho,\tau)]\cap \{v_{\pm}>0\}|^{1-\frac{1}{l_{0}}}.
	\end{align}

\end{lemma}

Based on Lamma \ref{lem003} and using the same argument as in Theorem 3.2 of \cite{HLZ202601}, we have the following global $L^{\infty}$ estimates.
\begin{lemma}\label{LEMZ01}
For any weak solution $u$ of \eqref{PO001} with $\Omega_T=Q(1,1)$, we have

$(i)$ for the Dirichlet case,
\begin{align*}
	\|u\|_{L^{\infty}(Q(1,1))}\leq\|g\|_{L^{\infty}(\partial B_{1})}+\|\phi\|_{L^{\infty}(B_{1})}+\overline{C}\|\psi\|_{L^{\frac{pl_{0}}{p-1}}(Q(1,1))}^{\frac{n+p}{(p-1)(n+2)}};
\end{align*}

$(ii)$ for homogeneous Neumann case,
\begin{align*}
	\|u\|_{L^{\infty}(Q(1,1))}\leq\|\phi\|_{L^{\infty}(B_{1})}+\overline{C}\|\psi\|_{L^{\frac{pl_{0}}{p-1}}(Q(1,1))}^{\frac{n+p}{(p-1)(n+2)}},
\end{align*}
where $\overline{C}=\overline{C}(n,p)$ and $l_{0}>\frac{n+p}{p}$.
\end{lemma}

Recall the classical Poisson kernel as follows:
\begin{align*}
	P(x,\xi)=\frac{1-|x|^{2}}{|\mathbb{S}^{n-1}||x-\xi|^{n}},
	\quad x\in B_{1},\;\,\xi\in\partial B_{1}.
\end{align*}
For $g\in L^{\infty}(\partial B_{1})\cap W^{1-\frac{1}{p},p}(\partial B_{1})$, define
\begin{align}\label{POISSON01}
	G_{0}(x)=\int_{\partial B_{1}}P(x,\xi)g(\xi)dS.
\end{align}
Then we have $G_{0}\in C^{\infty}(B_{1})$, $\Delta G_{0}=0$ in $B_{1}$, and
\begin{align}\label{POISSON02}
\begin{cases}	
	\|G_{0}\|_{L^{\infty}(B_{1})}\leq\|g\|_{L^{\infty}(\partial B_{1})},\quad \mathrm{Tr}G_{0}=g,\\
		\|G_{0}\|_{W^{1,p}(B_{1})}\leq C(n,p)\|g\|_{W^{1-\frac{1}{p},p}(\partial B_{1})},
\end{cases}	
\end{align}
where $\mathrm{Tr}$ is the trace operator defined by \eqref{TRACE}.\ The harmonicity of $G_0$ follows by differentiation under the
integral sign, see \cite[pp.\ 111--112]{ABR2001}. The $L^\infty$ estimate follows from \cite[Theorem 6.4(b)]{ABR2001}. The Sobolev estimate follows from the Poisson-integral characterizations in \cite[pp.\ 66--69, (2.1), (3.1), and Proposition 4.1]{G1974}. The identity $\operatorname{Tr}G_0=g$ then follows by smooth approximation and the continuity of the trace operator.

We next establish the following global $L^{p}_{t}W^{1,p}$ estimates. 
\begin{theorem}\label{thm03}
Assume that $u$ solves \eqref{PO001} weakly in $Q(1,1)$.\ Then we have

$(i)$ for the Dirichlet case,
\begin{align}\label{E05}
&\|u\|_{L^{p}((-1,0);W^{1,p}(B_{1}))}\notag\\
&\leq C\left(\big(\|\phi\|_{L^{\infty}(B_{1})}+\|g\|_{L^{\infty}(\partial B_{1})}\big)^{\frac{2}{p}}+\|u\|_{L^{\infty}(Q(1,1))}+\|g\|_{L^{\infty}(\partial B_{1})}\right)\notag\\
&\quad+C\Big(\|g\|_{W^{1-\frac{1}{p},p}(\partial B_{1})}+\|\psi\|_{L^{\frac{p}{p-1}}(Q(1,1))}^{\frac{1}{p-1}}\Big);
\end{align}

$(ii)$ for homogeneous Neumann case,
\begin{align}\label{E06}
&\|u\|_{L^{p}((-1,0);W^{1,p}(B_{1}))}\leq C\Big(\|\phi\|_{L^{\infty}(B_{1})}^{\frac{2}{p}}+\|u\|_{L^{\infty}(Q(1,1))}+\|\psi\|_{L^{\frac{p}{p-1}}(Q(1,1))}^{\frac{1}{p-1}}\Big),
\end{align}
where $C=C(n,p,L_{1},L_{2},L_{3})$.
	
\end{theorem}	

\begin{proof}
The proof is divided into the following two parts.

\noindent{\bf Part 1.}\ \emph{The Dirichlet case.}\ By testing \eqref{PO001} with $v=u-G_{0}$, justified by
a standard Steklov averaging argument, we obtain that for $s\in(-1,0]$,
\begin{align}\label{E03}
&\frac{1}{2}\|v(\cdot,s)\|_{L^{2}(B_{1})}^{2}+\int_{-1}^{s}\int_{B_{1}}\mathbf{a}(x,\nabla u)\cdot\nabla u dxdt\notag\\
&=\frac{1}{2}\|\phi-G_{0}\|_{L^{2}(B_{1})}^{2}+\int_{-1}^{s}\int_{B_{1}}\big(\mathbf{a}(x,\nabla u)\cdot\nabla G_{0}+b(x,t,u)v\big)dxdt,
\end{align}	
where $G_{0}$ is given in \eqref{POISSON01}. From \eqref{E01} and \eqref{E02}, it follows from Young's inequality that $\mathbf{a}(x,\nabla u)\cdot\nabla u\geq L_{1}|\nabla u|^{p}$,
\begin{align*}
|\mathbf{a}(x,\nabla u)\cdot\nabla G_{0}|\leq L_{2}|\nabla u|^{p-1}|\nabla G_{0}|\leq\frac{L_{1}}{2}|\nabla u|^{p}+C|\nabla G_{0}|^{p},	
\end{align*}		
and
\begin{align*}
|b(x,t,u)v|\leq\psi|v|\leq C\big(|v|^{p}+\psi^{\frac{p}{p-1}}\big).	
\end{align*}		
Substituting these inequalities into \eqref{E03}, we deduce from \eqref{POISSON02} that
\begin{align*}
&\sup_{s\in(-1,0)}\|u(\cdot,s)-G_{0}\|^{2}_{L^{2}(B_{1})}+\|u\|^{p}_{L^{p}((-1,0);W^{1,p}(B_{1}))}\notag\\
&\leq C\Big(\|\phi-G_{0}\|_{L^{2}(B_{1})}^{2}+\|u-G_{0}\|_{L^{p}(Q(1,1))}^{p}+\|g\|_{W^{1-\frac{1}{p},p}(\partial B_{1})}^{p}+\|\psi^{\frac{p}{p-1}}\|_{L^{1}(Q(1,1))}\Big)\\	
&\leq C\left(\|\phi\|_{L^{\infty}(B_{1})}^{2}+\|g\|_{L^{\infty}(\partial B_{1})}^{2}+\big(\|u\|_{L^{\infty}(Q(1,1))}+\|g\|_{L^{\infty}(\partial B_{1})}\big)^{p}\right)\\
&\quad +\Big(\|g\|_{W^{1-\frac{1}{p},p}(\partial B_{1})}^{p}+\|\psi^{\frac{p}{p-1}}\|_{L^{1}(Q(1,1))}\Big),
\end{align*}	
which implies that \eqref{E05} holds.

\noindent{\bf Part 2.}\ \emph{The homogeneous Neumann case.}\ Applying a standard Steklov averaging argument to justify the choice of $u$ as a test function
in \eqref{PO001}, we deduce that for $s\in(-1,0]$,
\begin{align*}
&\frac{1}{2}\|u(\cdot,s)\|_{L^{2}(B_{1})}^{2}+\int_{-1}^{s}\int_{B_{1}}\mathbf{a}(x,\nabla u)\cdot\nabla u dxdt\notag\\
&=\frac{1}{2}\|\phi\|_{L^{2}(B_{1})}^{2}+\int_{-1}^{s}\int_{B_{1}}b(x,t,u)udxdt,
\end{align*}	
which, together with \eqref{E01}, \eqref{E02} and Young's inequality, shows that
\begin{align*}
&\sup_{s\in(-1,0)}\|u(\cdot,s)\|^{2}_{L^{2}(B_{1})}+\|u\|^{p}_{L^{p}((-1,0);W^{1,p}(B_{1}))}\notag\\
&\leq C\big(\|\phi\|_{L^{2}(B_{1})}^{2}+\|u\|_{L^{p}(Q(1,1))}^{p}+\|\psi^{\frac{p}{p-1}}\|_{L^{1}(Q(1,1))}\big).
\end{align*}	
This proves \eqref{E06}.
	
\end{proof}

\section{Temporal H\"{o}lder regularity}\label{SEC003}

Define
\begin{align*}
	\mathcal{Q}_{1}=\|u\|_{L^\infty(Q(1,1))},\quad\mathcal{Q}_{2}=\|\nabla u\|_{L^p(Q(1,1))},\quad \mathcal{Q}_{3}=\|\psi\|_{L^{p'}((-1,0);L^q(B_1))},
\end{align*}
and
\begin{equation}\label{Z21}
	\begin{aligned}
		\mathcal{Q}_{4}&=\begin{cases}
			\|g\|_{L^\infty(\partial B_1)},&\text{for the Dirichlet case},\\
			0,&\text{for homogeneous Neumann case},
		\end{cases}\\
		\mathcal{Q}_{5}&=\begin{cases}
			\mathcal{Q}_{3},&\text{for homogeneous Neumann case},\\
			0,&\text{for the Dirichlet case}.
		\end{cases}
	\end{aligned}
\end{equation}
Denote
\begin{align}\label{E89}
\mathscr{C}=\max\big\{\mathcal{Q}_{1}+\mathcal{Q}_{4},\mathcal{Q}_{3}^{\mu}(1+\mathcal{Q}_{2})^\gamma+\mathcal{Q}_{5}\big\},	
\end{align}	
where
\begin{align}\label{E86}
\gamma=\frac{\vartheta(2-p)_+}{2},\quad\mu=\frac{p}{p+\vartheta(p-2)_{+}}.
\end{align}	
We now establish the desired temporal H\"{o}lder estimate as follows.
\begin{theorem}\label{thm:unified}
	There is a constant $C=C(n,p,q,\vartheta,L_{3})>0$ such that for all $\delta\in(0,1)$ and $-1+\delta\leq s<t\leq0$,
	\begin{equation}\label{eq:main-infty}
		\|u(\cdot,t)-u(\cdot,s)\|_{L^{\infty}(B_{1})}\leq C\mathscr{C}\delta^{-\beta}|t-s|^{\alpha},
	\end{equation}
	where $\alpha,\beta$ and $\mathscr{C}$ are defined by \eqref{eq:unified-exponents} and \eqref{E89}, respectively.
\end{theorem}

\begin{remark}
	In light of \eqref{E13} below, we see that the constant $C$ in \eqref{eq:main-infty} diverges to infinity, as $p\rightarrow2$.
	
\end{remark}

\subsection{The proof of Theorem \ref{thm:unified}}

Set
\begin{equation}\label{eq:spatial-gap}
	d_{n,p}=\frac{1}{\theta_{\mathrm{c}}}-\frac{1}{p}
	=\frac{1}{2}+\frac{1}{n}-\frac{1}{p}>0,\quad \sigma_\vartheta=\frac{1}{2}-\frac{1}{q}+\vartheta d_{n,p},
\end{equation}
where $\theta_{\mathrm{c}}$ and $\vartheta$ are, respectively, given by \eqref{eq:q-critical} and \eqref{eq:theta}.\ We first establish spatial estimates with support gain as follows.
\begin{lemma}\label{lem:spatial}
	Let $u\in W^{1,p}(B_1)$ and $E=\{x\in B_1:u(x)\neq0\}$.\ Assume either $u\in W_0^{1,p}(B_1)$ or $|E|\leq|B_1|/4$. Then we have
	\begin{equation*}
		\|u\|_{L^{q'}(B_1)}\leq C(n) |E|^{\sigma_\vartheta}\|u\|_{L^2(B_1)}^{1-\vartheta}\|\nabla u\|_{L^p(B_1)}^{\vartheta},
	\end{equation*}
	where $q'=\frac{q}{q-1}$ and $\sigma_{\vartheta}$ is given in \eqref{eq:spatial-gap}.
\end{lemma}
\begin{proof}
	If $u$ has zero trace, we have from the Sobolev inequality that $\|u\|_{L^{2}(B_{1})}\leq C(n)\|\nabla u\|_{L^{\theta_{\mathrm{c}}}(B_{1})}$. For the other case, let $\bar{u}=|B_{1}|^{-1}\int_{B_1}u(x)dx$. Then the Sobolev--Poincar\'e inequality gives that
	\begin{align*}
		\|u-\bar{u}\|_{L^{2}(B_{1})}\leq C(n)\|\nabla u\|_{L^{\theta_{\mathrm{c}}}(B_{1})}.
	\end{align*}
	Since $|B_{1}\setminus E|\geq3|B_{1}|/4$, we have
	\begin{align*}
		|\bar{u}||B_{1}\setminus E|^{\frac{1}{2}}=\|u-\bar{u}\|_{L^2(B_{1}\setminus E)}\leq C(n)\|\nabla u\|_{L^{\theta_{\mathrm{c}}}(B_{1}\setminus E)}.
	\end{align*}
	Consequently, we obtain
	\begin{align}\label{eq:001}
		\|u\|_{L^{2}(B_{1})}\leq\|u-\bar{u}\|_{L^{2}(B_{1})}+|B_{1}|^{\frac{1}{2}}|\bar{u}|\leq C(n)\|\nabla u\|_{L^{\theta_{\mathrm{c}}}(B_{1})}.
	\end{align}
From H\"older's inequality, we derive
	\begin{align*}
		\|\nabla u\|_{L^{\theta_{\mathrm{c}}}(B_{1})} \leq |E|^{\frac{1}{\theta_{\mathrm{c}}}-\frac{1}{p}}\|\nabla u\|_{L^{p}(E)}.
	\end{align*}
	Inserting this into \eqref{eq:001}, we deduce
		\begin{equation}\label{eq:spatial-L2}
		\|u\|_{L^2(B_1)} \leq C(n) |E|^{d_{n,p}}\|\nabla u\|_{L^p(B_1)},
	\end{equation}
where $d_{n,p}$ is defined by \eqref{eq:spatial-gap}. From \eqref{eq:q-critical}, we see that $q>\max\{2,p'\}$ and $q'=\frac{q-1}{q}<2$. Then we have from H\"{o}lder's inequality and \eqref{eq:spatial-L2} that
	\begin{align*}
		\|u\|_{L^{q'}(B_{1})}\leq& |E|^{\frac{1}{q'}-\frac{1}{2}}\|u\|_{L^{2}(B_{1})}\leq |E|^{\frac{1}{2}-\frac{1}{q}}\|u\|_{L^{2}(B_{1})}^{1-\vartheta}\|u\|_{L^{2}(B_{1})}^{\vartheta}\notag\\
		\leq&C(n) |E|^{\sigma_\vartheta}\|u\|_{L^2(B_1)}^{1-\vartheta}\|\nabla u\|_{L^p(B_1)}^{\vartheta}.
	\end{align*}

\end{proof}

For $a\geq1$ and $\xi\in\mathbb R^n$, define
\begin{equation}\label{eq:dissipation-density}
	\Phi_{p}(\xi,a)=
	\begin{cases}
		\displaystyle\frac{|\xi|^2}{a^{2-p}},&\text{for } p\in(\theta_{\mathrm{c}},2),\\
		|\xi|^p,&\text{for } p\in(2,\infty).
	\end{cases}
\end{equation}
Let $-1\leq r<s\leq0$, $\ell=s-r\leq1$, and suppose
\begin{align*}
	\begin{cases}	
		v\in C([r,s];L^2(B_1))\cap L^p((r,s);W^{1,p}(B_1)),\\
		\mathbf J\in L^{p'}(B_1\times(r,s);\mathbb R^n),\quad F\in L^{p'}((r,s);L^q(B_1)).
	\end{cases}
\end{align*}
Consider
\begin{equation}\label{eq:window-equation}
	\partial_\tau v-\mathrm{div}\mathbf J=F,\quad
	\mathbf{J}\cdot\nabla v\geq L\Phi_{p}(\nabla v,\mathfrak g).
\end{equation}
Here $L>0$, $\mathfrak g\geq1$ lies in $L^p(B_1\times(r,s))$ for $\theta_{\mathrm{c}}<p<2$, while $\mathfrak g\equiv1$ for $p>2$. Define
\begin{equation}\label{eq:window-data}
	\mathcal{K}_{F}=\|F\|_{L^{p'}((r,s);L^q(B_1))},\;\;	\Lambda=\begin{cases}
		\big(1+\|\mathfrak g\|_{L^p(B_1\times(r,s))}\big)^{\frac{2-p}{2}},&\text{if }p\in(\theta_{\mathrm{c}},2),\\
		1,&\text{if }p\in(2,\infty).
	\end{cases}
\end{equation}

\begin{lemma}\label{lem:window}
	Let $k_0\geq0$ and
	$|v(\cdot,r)|\leq k_0$. Either the weak equation of \eqref{eq:window-equation} holds for zero-trace spatial tests and $|\mathrm{Tr} v|\leq k_0$ on the lateral boundary, or it holds for unrestricted spatial tests with zero Neumann flux. For every $\tau\in[r,s]$, the Dirichlet case satisfies
	\begin{equation}\label{eq:window-D}
		\|v(\cdot,\tau)\|_{L^\infty(B_1)}
		\leq k_0+C\mathcal{K}_{F}^\mu\Lambda^{\mu\vartheta}\ell^\beta,
	\end{equation}
	and the homogeneous Neumann case satisfies
	\begin{equation}\label{eq:window-N}
		\|v(\cdot,\tau)\|_{L^\infty(B_1)}
		\leq k_0+C\big(\mathcal{K}_{F}\ell^{1/p}+\mathcal{K}_{F}^\mu\Lambda^{\mu\vartheta}\ell^\beta\big),
	\end{equation}
	where $\vartheta,\beta,\mu$ and $\mathcal{K}_{F},\Lambda$ are, respectively, defined in \eqref{eq:theta}--\eqref{eq:unified-exponents}, \eqref{E86} and \eqref{eq:window-data}, and the constant $C$ depends only on $n,p,q,\vartheta,L$.
\end{lemma}

\begin{proof}
	\noindent{\bf Step 1.} For a constant $k\geq k_0$, set $v_{k}=(v-k)_{+}$. Using a standard Steklov averaging argument to justify the choice of $v_{k}$ as a test function in
	\eqref{eq:window-equation}, we obtain that for $\tau\in[r,s]$,
	\begin{equation}\label{eq:exact-energy}
		\frac{1}{2}\|v_k(\cdot,\tau)\|_{L^{2}(B_{1})}^2+\int_{r}^{\tau}\int_{B_1}\mathbf J\cdot\nabla v_k dxd\tau=\int_{r}^{\tau}\int_{B_1}Fv_k dxd\tau.
	\end{equation}
	Define
	\begin{equation*}
		\mathcal{E}_{k}=
		\operatorname*{ess\,sup}_{r<\tau<s}\|v_{k}(\cdot,\tau)\|_{L^{2}(B_{1})}^2
		+\int_r^s\int_{B_1}\Phi_{p}(\nabla v_{k},\mathfrak g)dxd\tau,
		\end{equation*}
and
\begin{equation*}		
		\mathcal{S}_k=\int_r^s\int_{B_1}|F|v_{k}dxd\tau.
\end{equation*}
	Combining \eqref{eq:window-equation} and \eqref{eq:exact-energy}, it follows from H\"{o}lder's inequality that
	\begin{equation}\label{eq:energy-source}
		\mathcal{E}_{k}\leq(2+L^{-1})\mathcal{S}_{k}\leq (2+L^{-1})\mathcal{K}_{F}\|v_{k}\|_{L^p((r,s);L^{q'}(B_1))},
	\end{equation}
	where $\mathcal{K}_{H}$ is defined by \eqref{eq:window-data}.\ When $\theta_{\mathrm{c}}<p<2$, using H\"older's inequality, we obtain
	\begin{align*}
		\int_r^s\int_{B_1}|\nabla v_{k}|^pdxd\tau=&\int_r^s\int_{B_1}\bigg(\frac{|\nabla v_{k}|^2}{\mathfrak g^{2-p}}\bigg)^{p/2}
		\mathfrak g^{p(2-p)/2}dxd\tau\\
		\leq&\mathcal{E}_{k}^{p/2}\left(\int_r^s\int_{B_1}\mathfrak g^pdxd\tau\right)^{(2-p)/2}.
	\end{align*}
Then it follows that $\|\nabla v_{k}\|_{L^p(B_{1}\times(r,s))}\leq\Lambda\mathcal{E}_{k}^{1/2}$.\ When $p>2$, we directly have $\|\nabla v_{k}\|_{L^p}\leq\mathcal{E}_{k}^{1/p}$. Thus, both cases become
	\begin{equation}\label{eq:unified-gradient}
		\|\nabla v_{k}\|_{L^p(B_1\times(r,s))}\leq\Lambda\mathcal{E}_{k}^{1/p_\sharp},\quad p_\sharp=\max\{p,2\}.
	\end{equation}
	
	\noindent{\bf Step 2.} For the Dirichlet case, every $v_{k}$ already has zero trace. For the Neumann case, we need to arrange a large zero set before applying Lemma \ref{lem:spatial}.\ For $\tau\in[r,s]$, denote $Y_0(\tau)=\|v_{k_{0}}(\cdot,\tau)\|_{L^{2}(B_{1})}^2$.\ From \eqref{eq:exact-energy} and nonnegative dissipation, we deduce
	\begin{align*}
		Y_0'(\tau)\leq2\|F(\cdot,\tau)\|_2\sqrt{Y_0(\tau)},\quad Y_0(r)=0.
	\end{align*}
	For $\epsilon>0$, dividing by $2\sqrt{Y_0+\epsilon}$, integrating, and then letting $\epsilon\rightarrow0$, we obtain
	\begin{align}\label{eq:L2-seed}
		\|v_{k_{0}}(\cdot,\tau)\|_{L^{2}(B_{1})}\leq&\int_r^s\|F(\cdot,t)\|_{L^{2}(B_{1})}dt\leq|B_1|^{\frac{1}{2}-\frac{1}{q}} \int_r^s\|F(\cdot,t)\|_{L^{q}(B_{1})}dt\notag\\
		\leq& |B_1|^{1/2-1/q}\mathcal{K}_{F}\ell^{1/p}.
	\end{align}
	Let
	\begin{equation*}
		k_{\mathrm b}=
		\begin{cases}
			k_0,&\text{for the Dirichlet case},\\
			k_0+2|B_1|^{-1/q}\mathcal{K}_{F}\ell^{1/p},&\text{for homogeneous Neumann case}.
		\end{cases}
	\end{equation*}
	For the Neumann case and every $k\geq k_{\mathrm b}$, Chebyshev's inequality and \eqref{eq:L2-seed} imply
	\begin{equation*}
		|\{v(\cdot,\tau)>k\}|
		\leq
		\frac{|B_1|^{1-2/q}\mathcal{K}_{F}^2\ell^{2/p}}
		{\bigl(2|B_1|^{-1/q}\mathcal{K}_{F}\ell^{1/p}\bigr)^2}
		=\frac{|B_1|}{4}.
	\end{equation*}
	Hence Lemma~\ref{lem:spatial} applies at every subsequent level in either boundary case. 
	
	For $k\geq k_{\mathrm b}$, denote
	\begin{equation*}
		Y(k)=\operatorname*{ess\,sup}_{r<\tau<s}|\{x\in B_1:v(x,\tau)>k\}|.
	\end{equation*}
Then Lemma \ref{lem:spatial} gives that for $\tau\in[r,s],$
	\begin{align*}
		\|v_{k}(\cdot,\tau)\|_{L^{q'}(B_{1})}\leq C Y(k)^{\sigma_\vartheta}\|v_{k}(\cdot,\tau)\|_{L^{2}(B_{1})}^{1-\vartheta}\|\nabla v_{k}(\cdot,\tau)\|_{L^{p}(B_{1})}^\vartheta,
	\end{align*}
	where $\sigma_{\vartheta}$ is defined by \eqref{eq:spatial-gap}.\ It then follows from H\"{o}lder's inequality and \eqref{eq:unified-gradient} that
	\begin{align}\label{eq:mixed-calculation}
		\|v_{k}\|_{L^p((r,s);L^{q'}(B_1))}
		\leq& C Y(k)^{\sigma_\vartheta}
		\mathcal{E}_{k}^{(1-\vartheta)/2}
		\left(\int_r^s\|\nabla v_{k}(\cdot,\tau)\|_{L^{p}(B_{1})}^{p\vartheta}d\tau\right)^{1/p}
		\notag\\
		\leq& C Y(k)^{\sigma_\vartheta}
		\mathcal{E}_{k}^{(1-\vartheta)/2}\ell^{(1-\vartheta)/p}
		\left(\int_r^s\|\nabla v_{k}(\cdot,\tau)\|_{L^{p}(B_{1})}^p d\tau\right)^{\vartheta/p}\notag\\
		\leq& C\Lambda^\vartheta\ell^{b_\vartheta}
		Y(k)^{\sigma_\vartheta}\mathcal{E}_{k}^{e_{p,\vartheta}},
	\end{align}
	where
	\begin{equation*}
		b_\vartheta=\frac{1-\vartheta}{p},\quad
		e_{p,\vartheta}=\frac{1-\vartheta}{2}+\frac{\vartheta}{p_\sharp}
		=\frac{1}{2}-\frac{\vartheta(p-2)_+}{2p}\in(0,1).
	\end{equation*}
Here $p_{\sharp}$ is given by \eqref{eq:unified-gradient}.	Substituting \eqref{eq:mixed-calculation} into \eqref{eq:energy-source}, it follows that
	\begin{align*}
		\mathcal{E}_{k}\leq C\mathcal{K}_{F}\Lambda^\vartheta\ell^{b_\vartheta}
		Y(k)^{\sigma_\vartheta}\mathcal{E}_{k}^{e_{p,\vartheta}}.
	\end{align*}
	Note that
	\begin{align*}
		&1-e_{p,\vartheta}=\frac{p+\vartheta(p-2)_+}{2p},\quad
		\frac{b_\vartheta}{1-e_{p,\vartheta}}=\frac{1-\vartheta}{p}\frac{2p}{p+\vartheta(p-2)_{+}}=2\beta,\\
		&\sigma_\vartheta-(1-e_{p,\vartheta})=\vartheta\left(d_{n,p}-\frac{(p-2)_+}{2p}\right)-\frac{1}{q}
		=\frac{\vartheta}{q_{\mathrm{c}}}-\frac{1}{q},\quad\frac{1}{1-e_{p,\vartheta}}=2\mu.
	\end{align*}
	Define
	\begin{equation*}
		\kappa=2\mu\left(\frac{\vartheta}{q_{\mathrm{c}}}-\frac{1}{q}\right)>0.
	\end{equation*}
	Therefore, we have
	\begin{equation}\label{eq:common-level-bound}
		\mathcal{E}_{k}\leq C_*
		\mathcal{K}_{F}^{2\mu}\Lambda^{2\mu\vartheta}\ell^{2\beta}
		Y(k)^{1+\kappa},
	\end{equation}
	where $C_*>0$ depends only on $n,p,q,\vartheta,L$.
	
	\noindent{\bf Step 3.} For $j>k\geq k_{\mathrm b}$, one has $v_{k}\geq j-k$ on $\{v>j\}$. Hence, we have
	\begin{align*}
		(j-k)^2Y(j)\leq
		\operatorname*{ess\,sup}_{r<\tau<s}\|v_{k}(\cdot,\tau)\|_2^2
		\leq\mathcal{E}_{k},
	\end{align*}
	which, combined with \eqref{eq:common-level-bound}, leads to that 
	\begin{equation*}
		Y(j)\leq\frac{C_*\mathcal{G}^2}{(j-k)^2}Y(k)^{1+\kappa},\quad \mathcal{G}:=\mathcal{K}_{F}^\mu\Lambda^{\mu\vartheta}\ell^\beta.
	\end{equation*}
	Choose
	\begin{equation*}
		\mathcal{M}=2C_*^{1/2}4^{1/(2\kappa)}|B_{1}|^{\kappa/2}\mathcal{G},\quad
		k_{i}=k_{\mathrm b}+\mathcal{M}(1-2^{-i}),\quad i\geq0.
	\end{equation*}
	Then $k_{i+1}-k_{i}=\mathcal{M}2^{-i-1}$. Writing $Z_{i}=Y(k_{i})/|B_{1}|$, we obtain
	\begin{align*}
		Z_{i+1}
		&\leq\frac{4C_*\mathcal{G}^2|B_{1}|^\kappa}{\mathcal{M}^2}4^{i}Z_{i}^{1+\kappa}=4^{-1/\kappa}4^{i}Z_{i}^{1+\kappa}.
	\end{align*}
	The induction $Z_{i}\leq4^{-i/\kappa}$ holds, since
	\begin{align*}
		4^{-1/\kappa}4^{i}\big(4^{-i/\kappa}\big)^{1+\kappa}=4^{-(i+1)/\kappa}.
	\end{align*}
	Thus $Y(k_{i})\to0$, and $v\leq k_{\mathrm b}+\mathcal{M}$. Applying the same argument to $-v$ with flux $-\mathbf J$ and source $-F$, it follows that $|v|\leq k_{\mathrm b}+\mathcal{M}$. This proves \eqref{eq:window-D}--\eqref{eq:window-N}.
	
\end{proof}

We next give the proof of Theorem \ref{thm:unified} with the aid of Lemma \ref{lem:window}.
\begin{proof}[Proof of Theorem \ref{thm:unified}]
	For any $0<\delta\ll1$ and $-1+\delta\leq s<t\leq0,$ let 
	\begin{align*}
		h=t-s,\quad h\leq\ell \leq\frac{\delta}{2},\quad r=s-\ell,\quad \lambda=\Big(1+\frac{h}{\ell}\Big)^{\frac{1}{p-2}}.	
	\end{align*}	
	Since $0<\frac{h}{\ell}\leq1,$ it follows from a direct computation that
	\begin{align}\label{eq:lambda-bounds}
		\min\big\{1,2^{\frac{1}{p-2}}\big\}\leq\lambda\leq\max\big\{1,2^{\frac{1}{p-2}}\big\},\quad |\lambda-1|\leq C_{p}\frac{h}{\ell},	
	\end{align}	
	where
	\begin{align}\label{E13}
		C_{p}=\max\bigg\{\frac{1}{|p-2|},\big|2^{\frac{1}{p-2}}-1\big|\bigg\}\rightarrow\infty,\quad\text{as }p\rightarrow2.
	\end{align}
	
	For $\tau\in[r,s]$, write
	\begin{align*}
		T_{\lambda}(\tau)=r+\lambda^{p-2}(\tau-r),\quad\hat{u}(x,\tau)=\lambda u(x,T_{\lambda}(\tau)).
	\end{align*}
	Then we have
	\begin{align}\label{E10}
		T_{\lambda}(r)=r,\quad T_{\lambda}(s)=t,\quad\hat{u}(x,s)=\lambda u(x,t),\quad	\hat{u}(x,r)=\lambda u(x,r).
	\end{align}	
	Note that $\hat{u}$ solves
	\begin{equation*}
		\begin{cases}
			\partial_{\tau}\hat{u}-\mathrm{div}\big(\mathbf{a}(x,\nabla \hat{u})\big)=\lambda^{p-1}b(x,T_{\lambda}(\tau),\lambda^{-1}\hat{u}),&\text{in }B_{1}\times(r,s],\\
			\hat{u}=\lambda g(x)\;\,\text{or}\;\,
			\mathbf{a}(x,\nabla \hat{u})\cdot\nu=0,&\text{on }\partial  B_{1}\times(r,s],\\
			\hat{u}(x,r)=\lambda u(x,r),&\text{in } B_{1}.
		\end{cases}
	\end{equation*}
	Denote
	\begin{align*}
		\mathbf J_\lambda=\mathbf{a}(x,\nabla \hat{u})-\mathbf{a}(x,\nabla u),\quad F_{\lambda}(x,\tau)=\lambda^{p-1}b(x,T_{\lambda}(\tau),\lambda^{-1}\hat{u})-b(x,\tau,u(x,\tau)).	
	\end{align*}	
	Then $v(x,\tau):=\hat{u}(x,\tau)-u(x,\tau)$ satisfies
	\begin{equation*}
		\begin{cases}
			\partial_{\tau}v-\mathrm{div}\mathbf{J}_\lambda=F_{\lambda}(x,\tau),&\text{in }B_{1}\times(r,s],\\
			v=(\lambda-1)g(x)\;\,\text{or}\;\,
			\mathbf{J}_\lambda\cdot\nu=0,&\text{on }\partial  B_{1}\times(r,s],\\
			v(x,r)=(\lambda-1) u(x,r),&\text{in } B_{1}.
		\end{cases}
	\end{equation*}
	From \eqref{E32} and \eqref{eq:lambda-bounds}, we obtain
	\begin{align*}
		\mathbf{J}_{\lambda}\cdot\nabla v\geq L_{3}\Phi_{p}(\nabla v,\mathfrak{g}_{\lambda}),\quad\mathcal{K}_{F_\lambda}=\|F_{\lambda}\|_{L^{p'}((r,s);L^q(B_1))}\leq C(p)\mathcal{Q}_{3},
	\end{align*}	
	where $\Phi_{p}$ is defined in \eqref{eq:dissipation-density}, and
	\begin{align*}
		\mathfrak g_\lambda=1+|\nabla\hat{u}|+|\nabla u|,\quad \text{for }\theta_{\mathrm{c}}<p<2.
	\end{align*}
	Using \eqref{eq:lambda-bounds} again, we deduce
	\begin{align*}
		\int_r^s\int_{B_1}\mathfrak g_\lambda^pdxd\tau
		&\leq3^{p-1}\big(|B_1|\ell+(1+\lambda^2)\mathcal{Q}_{2}^p\big)\leq C(n,p)(1+\mathcal{Q}_{2})^p.
	\end{align*}
	Since $\mu=1$ for $\theta_{\mathrm{c}}<p<2$, we have
	\begin{align*}
		\Lambda^{\mu\vartheta}\leq C(n,p,\vartheta)(1+\mathcal{Q}_{2})^{\vartheta(2-p)/2}.
	\end{align*}
	Therefore, we obtain that for all $p\in(\theta_{\mathrm{c}},2)\cup(2,\infty),$
	\begin{equation*}
		\mathcal K_{F_\lambda}^{\mu}\Lambda^{\mu\vartheta}
		\leq C\mathcal{Q}_{3}^{\mu}(1+\mathcal{Q}_{2})^\gamma.
	\end{equation*}
	Denote
	\begin{equation*}
		\mathcal{A}=\mathcal{Q}_{3}^{\mu}(1+\mathcal{Q}_{2})^\gamma+\mathcal{Q}_{5}.
	\end{equation*}
	Then it follows from Lemma \ref{lem:window} that
	\begin{equation}\label{eq:z-control}
		\|v(\cdot,s)\|_{L^{\infty}(B_{1})}\leq C\big((\mathcal{Q}_{1}+\mathcal{Q}_{4})h\ell^{-1}+\mathcal{A}\ell^\beta\big).
	\end{equation}
	Combining \eqref{eq:lambda-bounds} and \eqref{E10}--\eqref{eq:z-control}, we have
	\begin{align*}
		\|u(\cdot,t)-u(\cdot,s)\|_{L^\infty(B_1)}\leq &\lambda^{-1}\|v(\cdot,s)\|_{L^{\infty}(B_{1})}+\lambda^{-1}|1-\lambda|\|u(\cdot,s)\|_{L^{\infty}(B_{1})}\notag\\
		\leq& C\big((\mathcal{Q}_{1}+\mathcal{Q}_{4})h\ell^{-1}+\mathcal{A}\ell^\beta\big).
	\end{align*}
	Therefore, on one hand, if $0<h\leq\big(\frac{\delta}{2})^{1+\beta}$, by taking $\ell=h^{\frac{1}{1+\beta}},$ we have
	\begin{align*}
		\|u(\cdot,t)-u(\cdot,s)\|_{L^{\infty}(B_{1})}\leq C\max\big\{\mathcal{Q}_{1}+\mathcal{Q}_{4},\mathcal{A}\big\}|t-s|^{\frac{\beta}{1+\beta}}.	
	\end{align*}	
	On the other hand, if $h>\big(\frac{\delta}{2})^{1+\beta}$, we have
	\begin{align*}
		\|u(\cdot,t)-u(\cdot,s)\|_{L^{\infty}(B_{1})}\leq C\max\big\{\mathcal{Q}_{1}+\mathcal{Q}_{4},\mathcal{A}\big\}(2/\delta)^{\beta}|t-s|^{\frac{\beta}{1+\beta}}.	
	\end{align*}	
	This implies that Theorem \ref{thm:unified} holds.

\end{proof}

We are now ready to prove Theorem \ref{thm10}.
\begin{proof}[Proof of Theorem \ref{thm10}]
It follows from Lemma \ref{LEMZ01}, Theorem \ref{thm03} and Theorem \ref{thm:unified} that Theorem \ref{thm10} holds.

\end{proof}

\section{Counterexamples to temporal regularity}\label{SEC005}

We first prove Proposition \ref{prop:critical}.
\begin{proof}[Proof of Proposition \ref{prop:critical}]
	\noindent{\bf Step 1.}\	\emph{A compactly supported profile with continuous flux.}
	Let
	\begin{equation*}
		\Phi(y)=\big(1-|y|^{p'}\big)_+^{p'+1},\quad \text{for }y\in\mathbb{R}^{n},
	\end{equation*}
	where $r_+=\max\{r,0\}$ and $p'=\frac{p}{p-1}$.\ Then we have $\Phi\in C_{0}^1(\mathbb R^n)$,
	$0\leq\Phi\leq1$, $\Phi(0)=1$, and
	\begin{align*}
		\nabla\Phi(y)
		&=-p'(p'+1)|y|^{p'-2}(1-|y|^{p'})^{p'}y,\quad\text{for } 0<|y|<1,
	\end{align*}
	which gives that
	\begin{equation}\label{eq:profile-flux}
		|\nabla\Phi(y)|^{p-2}\nabla\Phi(y)
		=-c_p\big(1-|y|^{p'}\big)_+^p y,\quad  c_p=[p'(p'+1)]^{p-1}.
	\end{equation}
	The flux at zero gradient is defined to be zero and the right-hand side of \eqref{eq:profile-flux} is $C^1$ in $\mathbb R^n$. In particular, no surface measure occurs on $|y|=1$. Its divergence is the continuous bounded function
	\begin{equation}\label{eq:G}
		G(y):=\mathrm{div}\big(|\nabla\Phi|^{p-2}\nabla\Phi\big)
		=-c_p\big(
		n(1-|y|^{p'})_+^p
		-pp'|y|^{p'}(1-|y|^{p'})_+^{p-1}
		\big).
	\end{equation}
	Here both terms are zero when $|y|\geq1$, and $G(0)=-nc_p$.
	
\smallskip
	\noindent{\bf Step 2.}\	\emph{The explicit solution and boundary data.}
	For $0<\tau\leq\frac{1}{2}$, define
	\begin{equation}\label{eq:scales}
		\Lambda(\tau)=\log\frac{e}{\tau},\quad
		a(\tau)=\Lambda(\tau)^{-1},\quad
		r(\tau)=\big(\tau a(\tau)^{p-2}\big)^{1/p},
	\end{equation}
where $e$ is the base of the natural logarithm.\ Thus $a(\tau)\to0$ and $r(\tau)\to0$ as $\tau\rightarrow0$. For simplicity, denote $t_{0}=-\frac{1}{2}$. For $x\in\overline{B}_{1}$, let
	\begin{equation*}
		u(x,t)=
		\begin{cases}
			1,&-1\leq t\leq t_0,\\
			1+a(t-t_0)\Phi\big(\frac{x}{r(t-t_0)}\big),
			&\;\,t_0<t\leq0.
		\end{cases}
	\end{equation*}
	In view of $p>1$, we deduce
	\begin{align*}
		r(\tau)^p
		&=\tau\Lambda(\tau)^{2-p}
		\leq\tau\Lambda(\tau)
		\leq\frac{\log(2e)}{2}<1,
	\end{align*}
where we also used the fact of $\frac{d}{d\tau}[\tau\log(e/\tau)]=\log(1/\tau)>0$. Consequently, the perturbation is supported in the fixed ball $\overline{B}_{r_*}$, where $r_*=[\log(2e)/2]^{1/p}<1$. It is identically zero in a fixed neighborhood of $\partial B_1$.\ Then we have $u=1$ and $\nabla u=0$ there, which verifies both alternative boundary conditions.\ The initial datum is the constant $1$, so it has Dirichlet trace $1$ and zero normal flux. Observe that 
\begin{equation}\label{eq:uniform-continuity}
		\|u(\cdot,t_0+\tau)-1\|_{L^\infty(B_1)}=a(\tau)\longrightarrow0,\quad\text{as }\tau\rightarrow0.
\end{equation}
This proves joint continuity on the slice $t=t_0$. Hence $u\in C(\overline{B}_{1}\times[-1,0])$.
	
\smallskip
\noindent{\bf Step 3.}\	\emph{The source and the exact scaling identities.}
	For $t>t_0$, write $\tau=t-t_0$ and $y=x/r(\tau)$. Differentiating
	\eqref{eq:scales}, we obtain
	\begin{equation*}
		a'(\tau)=\frac{a(\tau)^2}{\tau},\quad
		\frac{r'(\tau)}{r(\tau)}
		=\frac{1+(p-2)a(\tau)}{p\tau}.
	\end{equation*}
	Then we have
	\begin{align}\label{eq:time-derivative}
		\partial_tu(x,t)
		&=a'(\tau)\Phi(y)-a(\tau)\frac{r'(\tau)}{r(\tau)}y\cdot\nabla\Phi(y)
		\notag\\
		&=\frac{a(\tau)}{\tau}\Big(a(\tau)\Phi(y)-\frac{1+(p-2)a(\tau)}{p}y\cdot\nabla\Phi(y)\Big),
	\end{align}
	and
	\begin{align*}
		\nabla u(x,t)=\frac{a(\tau)}{r(\tau)}\nabla\Phi(y),\quad
		\mathrm{div}\big(|\nabla u|^{p-2}\nabla u\big)(x,t)
		=\frac{a(\tau)^{p-1}}{r(\tau)^p}G(y)
		=\frac{a(\tau)}{\tau}G(y),
	\end{align*}
	where $G$ is given by \eqref{eq:G}. For $0\leq a\leq a_*:=1/\log(2e)$, define
	\begin{equation}\label{eq:source-profile}
		\mathcal R_a(y)
		=a\Phi(y)-\frac{1+(p-2)a}{p}y\cdot\nabla\Phi(y)-G(y),
	\end{equation}
	and
	\begin{equation}\label{eq:source}
		F(x,t)=
		\begin{cases}
			0,&-1\leq t\leq t_0,\\
			\displaystyle
			\frac{a(t-t_0)}{t-t_0}
			\mathcal R_{a(t-t_0)}\!\Big(\frac{x}{r(t-t_0)}\Big),
			&\;\, t_0<t\leq0.
		\end{cases}
	\end{equation}
	Equations \eqref{eq:time-derivative}--\eqref{eq:source} verify the pointwise equation for $t>t_0$ and it is also immediate for $t<t_0$. The functions $\mathcal R_a$ are uniformly bounded, continuous, and supported in $\overline{B}_{1}$. The source is time dependent and unbounded. In fact, at the center, we have
	\begin{equation*}
		F(0,t_0+\tau)
		=\frac{a(\tau)}{\tau}\bigl(a(\tau)+nc_p\bigr)\longrightarrow\infty,\quad\text{as }\tau\rightarrow0.
	\end{equation*}
	Thus no continuity or boundedness of $F$ on the closed cylinder $\overline{B}_{1}\times[-1,0]$ is asserted.
	
\smallskip
\noindent{\bf Step 4.}\ \emph{Integrability and the weak formulation.}
	Since $r(\tau)<1$, a spatial change of variables in \eqref{eq:source}
	yields the exact identity
	\begin{equation}\label{eq:source-norm}
		\|F(\cdot,t_0+\tau)\|_{L^q(B_1)}
		=\frac{a(\tau)}{\tau}r(\tau)^{n/q}
		\|\mathcal R_{a(\tau)}\|_{L^q(\mathbb R^n)},
		\quad 1\leq q<\infty.
	\end{equation}
	At $q=n$, uniform boundedness and compact support of $\mathcal R_a$ give
	\begin{align*}
		\|F(\cdot,t_0+\tau)\|_{L^n(B_1)}^{p'}
		&\leq C\Big(\frac{a(\tau)}{\tau}r(\tau)\Big)^{p'}
		\leq C\tau^{-p'+p'/p}a(\tau)^{p'+(p-2)p'/p}
		\leq\frac{C}{\tau\Lambda(\tau)^2}.
	\end{align*}
	Substituting $v=\Lambda(\tau)$, we obtain
	\begin{equation*}
		\int_0^{1/2}\frac{d\tau}{\tau\Lambda(\tau)^2}
		=\int_{\log(2e)}^{\infty}\frac{dv}{v^2}
		=\frac{1}{\log(2e)}<\infty.
	\end{equation*}
	This proves $F\in L^{p'}((-1,0);L^n(B_1))$ for every $p>1$. The same calculation applied to \eqref{eq:time-derivative} shows that its right-hand side, extended by zero for $t\leq t_0$, also belongs to $L^{p'}((-1,0);L^n(B_1))$. 
	
	Set $C_\Phi=\int_{\mathbb R^n}|\nabla\Phi|^p\,dy<\infty$.
	The exact spatial energy is
	\begin{align*}
		\int_{B_1}|\nabla u(x,t_0+\tau)|^pdx=C_\Phi a(\tau)^p r(\tau)^{n-p}=C_\Phi\tau^{n/p-1}
		a(\tau)^{2+n(p-2)/p}.
	\end{align*}
	This is integrable in time, since for every $b\in\mathbb{R}$,
	\begin{align*}
		\int_0^{1/2}\tau^{n/p-1}\Lambda(\tau)^{-b}\,d\tau
		=e^{n/p}\int_{\log(2e)}^{\infty}e^{-(n/p)v}v^{-b}\,dv<\infty.
	\end{align*}
	Consequently, $u$ belongs to the asserted energy class.
	
	The $C^1$ flux in \eqref{eq:profile-flux} and the vanishing of $\Phi$ on its support boundary exclude singular moving-interface terms. Integrating by parts away from $t_0$ and passing to the limit using \eqref{eq:uniform-continuity}, it follows from the preceding integrability bounds that
	\begin{equation}\label{eq:weak}
		\int_{-1}^{0}\int_{B_1}\big(-u\partial_t\zeta+|\nabla u|^{p-2}\nabla u\cdot\nabla\zeta\big)dxdt=\int_{-1}^{0}\int_{B_1}F\zeta dxdt,
	\end{equation}
where $\zeta$ is smooth and compactly supported in time, with zero spatial trace in the Dirichlet case and no trace restriction in homogeneous Neumann case.\ The spatial boundary term vanishes because the flux is zero near $\partial B_1$.\ The same limiting argument identifies the expression in \eqref{eq:time-derivative}, extended by zero for $t\leq t_0$, as the weak time derivative. No jump term occurs at $t_0=-\frac{1}{2}$.
	
	Note that for every $p>1$,
	\begin{equation*}
		W^{1,p}(B_1)\hookrightarrow W^{1,1}(B_1)
		\hookrightarrow L^{n/(n-1)}(B_1).
	\end{equation*}
	The complete Sobolev norm is used, so no zero-mean or zero-trace condition is needed for this embedding. It then follows from H\"older's inequality that
	\begin{align*}
		\int_{-1}^{0}\int_{B_1}|F\zeta|\,dx\,dt
		&\leq \|F\|_{L^{p'}((-1,0);L^n(B_1))}
		\|\zeta\|_{L^p((-1,0);L^{n/(n-1)}(B_1))}\notag\\
		&\leq C(n,p)\|F\|_{L^{p'}((-1,0);L^n(B_1))}
		\|\zeta\|_{L^p((-1,0);W^{1,p}(B_1))}.
	\end{align*}
	Thus density extends \eqref{eq:weak} to time-regularized energy tests with spatial values in $W_0^{1,p}(B_1)\cap L^2(B_1)$ or $W^{1,p}(B_1)\cap L^2(B_1)$, respectively. Finally, $u\equiv1$ near $t=-1$ gives the strong $L^2(B_1)$ initial trace. Hence $u$ is an energy weak solution for all $p>1$.
	
\smallskip
\noindent{\bf Step 5.}\	\emph{Failure of any temporal H\"older exponent.}
	For $0<h<1/4$, choose $s_h=t_0+h$, $t_h=t_0+2h$, and set
	$L_h=\log(e/h)$. Since $\Phi(0)=1$, we have
	\begin{align*}
		u(0,t_h)-u(0,s_h)
		&=\frac{1}{L_h-\log2}-\frac{1}{L_h}
		=\frac{\log2}{L_h(L_h-\log2)}>0.
	\end{align*}
	For each fixed $h$, the difference of the two spatial slices is
	continuous in $x$, so its essential supremum is at least its value at
	$x=0$. Thus, for every fixed $\bar{\alpha}>0$, we have
	\begin{equation*}
		\frac{\|u(\cdot,t_h)-u(\cdot,s_h)\|_{L^\infty(B_1)}}{(t_h-s_h)^{\bar{\alpha}}}
		\geq\frac{\log2}{h^{\bar{\alpha}} L_h(L_h-\log2)}
		\longrightarrow\infty,
		\quad h\rightarrow0.
	\end{equation*}
	For any prescribed $0<\rho\leq\frac{1}{2}$, taking $2h<\rho$, we see that both times strictly inside $(t_0,t_0+\rho)$. This proves
	\eqref{eq:failure-claim}. 
	
\smallskip
\noindent{\bf Step 6.}\	\emph{The source is not in any higher spatial mixed class.}
	By \eqref{eq:source-profile}, we have $\mathcal R_a(0)=a+nc_p\geq nc_p$ for $0\leq a\leq a_*$.\ The continuous family $\mathcal R_a$ is uniformly continuous near $y=0$. Hence there exists $r_0\in(0,1)$, independent of $a$, such that $\mathcal R_a(y)\geq nc_p/2$ on $B_{r_0}$. For each $q\geq1$, we have
	\begin{align*}
		0<\bar{c}\leq\|\mathcal R_a\|_{L^q(\mathbb R^n)}\leq \overline{C}<\infty,
		\quad\text{for } 0\leq a\leq a_*,
	\end{align*}
	where the constants $\bar{c}$ and $\overline{C}$ depend on the fixed $n,p,q$, not on time. Combining this with \eqref{eq:source-norm}, it follows that
	\begin{equation*}
		\|F(\cdot,t_0+\tau)\|_{L^q(B_1)}^{p'}
		\sim
		\tau^{-1-\frac{1-n/q}{p-1}}
		\Lambda(\tau)^{-\frac{p+(p-2)n/q}{p-1}},
	\end{equation*}
	where $X\sim Y$ means $cY\leq X\leq CY$ for fixed positive constants $c$ and $C$. If $q>n$, denote $d=(1-n/q)/(p-1)>0$. For every $b\in\mathbb{R}$, we have
	\begin{align*}
		\int_0^{1/2}\tau^{-1-d}\Lambda(\tau)^{-b}d\tau=e^{-d}\int_{\log(2e)}^{\infty}e^{dv}v^{-b}dv=\infty.
	\end{align*}
	Thus $F\notin L^{p'}((-1,0);L^q(B_1))$ for every finite $q>n$.
\end{proof}

The preceding counterexample shows that forcing in $L^{p'}_tL^n_x$ need not guarantee any positive temporal H\"older exponent, even for continuous bounded energy solutions with compatible data.\ The slowly decaying amplitude $a(\tau)=1/\log(e/\tau)$ and the shrinking scale $r(\tau)^p=\tau a(\tau)^{p-2}$ preserve the mixed-norm bound while excluding every power-type temporal modulus in this space.\ The source in this construction is unbounded.

We next consider the stronger assumption of continuous bounded time-dependent forcing.\ The following construction shows that even under this assumption, the temporal Lipschitz endpoint in $L^\infty$ may fail, although the solutions retain temporal H\"older continuity in this space for every exponent below one.

\begin{proof}[Proof of Proposition \ref{prop:main}]
	\noindent{\bf Step 1.} Start with a time-independent positive reference function
	$V\in C^1(\overline{B}_{1})$ whose weak source
	\begin{align}\label{G0012}
		G=-\mathrm{div}(|\nabla V|^{p-2}\nabla V)
		\quad\text{belongs to }C(\overline{B}_{1}).
	\end{align}
Fix $x_0$ and a ball $U\Subset B_1$ about $x_0$ on which $V$ is smooth and $|\nabla V|\geq c_{\ast}>0$, where the symbol $\Subset$ denotes compact inclusion. Specific reference functions will be given in
	Step 4. Define 
	\begin{equation*}
		\begin{gathered}
			\mathcal{B}(\xi)=|\xi|^{p-2}I+(p-2)|\xi|^{p-4}\xi\otimes\xi,\quad \xi\neq0,\\
			 \mathcal{A}=\mathcal{B}(\nabla V(x_0)),\quad
			\operatorname{tr}\mathcal{A}
			=(n+p-2)|\nabla V(x_0)|^{p-2}>0.
		\end{gathered}
	\end{equation*}
	The eigenvalues of $\mathcal{B}(\xi)$ are $|\xi|^{p-2}$ and $(p-1)|\xi|^{p-2}$, so $\mathcal{A}$ is positive definite. Write
	\begin{equation}\label{eq:profile}
		\begin{gathered}
			y=x-x_0,\quad t_{0}=-\frac{1}{2},\quad S=|y|^4+(t-t_0)^2,\quad
			P(x,t)=t-t_0+\frac{|y|^2}{2\operatorname{tr}\mathcal{A}},\\
			H(s)=\sqrt{\log\frac{e^4}{s}},\quad W(x,t)=P(x,t)H(S)
			\quad(S>0),\quad W(x_0,t_0)=0,
		\end{gathered}
	\end{equation}
	where $e$ denotes the base of the natural logarithm.\ Only $x_0=0$ and $x_0=\frac{1}{2}e_1$ will be used.\ Here $e_{1}=(1,0,...,0).$\ On $\overline{B}_{1}\times[-1,0]$, we have $S\leq\frac{85}{16}<e^4$.\ So $H$ is well defined for every $S>0$ encountered here.\ Take $0\leq\chi,\eta\leq1$, with $\chi\in C_{0}^\infty(U)$ equal to one near $x_0$, and $\eta\in C_{0}^\infty((-\frac{3}{4},-\frac{1}{4}))$ equal to one near $t_0$. Set
	\begin{equation}\label{eq:u}
		w(x,t)=\chi(x)\eta(t)W(x,t),\quad u(x,t)=V(x)+\epsilon w(x,t).
	\end{equation}
	The parameter $\epsilon>0$ will be fixed sufficiently small and no limiting argument in $\epsilon$ is used to prove failure of Lipschitz continuity.
	
	For simplicity, write $\rho=S^{1/4}$. Near $(x_0,t_0)$, direct differentiation gives
	\begin{equation}\label{eq:bounds}
		|W|\leq C\rho^2H(S),\quad
		|\nabla W|\leq C\rho H(S),\quad
		|\nabla^{2}W|+|\partial_tW|\leq C H(S).
	\end{equation}
	In fact, $|P|\leq C S^{1/2}$, $|\nabla P|\leq C S^{1/4}$,
	$|\nabla S|\leq4S^{3/4}$, and $|\nabla^{2}S|\leq C S^{1/2}$. Moreover, we have
	\begin{equation*}
		H'(s)=-\frac{1}{2sH(s)},\quad
		H''(s)=\frac{2H(s)^2-1}{4s^2H(s)^3}.
	\end{equation*}
	These formulas imply \eqref{eq:bounds}.\ In particular, $w$ and its spatial gradient extend continuously across the exceptional point $(x_{0},t_{0})$.\ Pick $\epsilon$ so small that $u>0$ and $\epsilon\|\nabla w\|_{L^{\infty}}\leq\frac{c_{\ast}}{2}$ on $U\times[-1,0]$.
	
	\noindent{\bf Step 2.} For matrices, denote $A:B=\sum_{i,j}A_{ij}B_{ij}$. Let $\mathcal{L}_0=\mathcal{A}:\nabla^{2}$. From \eqref{eq:profile}, we have
	\begin{equation*}
		\partial_tP=1,\quad
		\nabla^{2}P=\frac{I}{\operatorname{tr}\mathcal{A}},\quad
		(\partial_t-\mathcal{L}_0)P=0.
	\end{equation*}
	This exact cancellation removes the unbounded factor $H(S)$ from the
	linear residual. Indeed, we have
	\begin{align}\label{eq:linear-residual}
		&(\partial_t-\mathcal{L}_0)W\notag\\
		&=P H'(S)(\partial_tS-\mathcal{L}_0S)
		-P H''(S)\nabla S\cdot\mathcal{A}\nabla S
		-2H'(S)\nabla P\cdot\mathcal{A}\nabla S.
	\end{align}
	Here $\partial_tS=2(t-t_0)$, $\nabla S=4|y|^2y$, and
	$\nabla^{2}S=8y\otimes y+4|y|^2I$. Thus each term in
	\eqref{eq:linear-residual} is bounded in absolute value by $C/H(S)$, and
	\begin{equation}\label{eq:linear-small}
		|(\partial_t-\mathcal{L}_0)W|\leq\frac{C}{H(S)}\longrightarrow0,
		\quad\text{as }(x,t)\to(x_0,t_0).
	\end{equation}
	
	For $(x,t)\neq(x_0,t_0)$, define $F=\partial_tu-\mathrm{div}_x(|\nabla u|^{p-2}\nabla u)$, and set $F(x_0,t_0)=G(x_{0})$, where $G$ is given by \eqref{G0012}. A direct calculation gives that
	\begin{align*}
		F(x,t)-G(x)
		&=\epsilon(\partial_tW-\mathcal{A}:\nabla^{2}W)
		-\epsilon[\mathcal{B}(\nabla V+\epsilon\nabla W)-\mathcal{A}]:\nabla^{2}W\notag\\
		&\quad-[\mathcal{B}(\nabla V+\epsilon\nabla W)-\mathcal{B}(\nabla V)]:\nabla^{2}V.
	\end{align*}
Note that
	\begin{align*}
		|\mathcal{B}(\nabla V+\epsilon\nabla W)-\mathcal{A}|\leq C(|y|+\epsilon|\nabla W|),\;\,
		|\mathcal{B}(\nabla V+\epsilon\nabla W)-\mathcal{B}(\nabla V)|\leq C\epsilon|\nabla W|.
	\end{align*}
Combining these bounds with \eqref{eq:bounds} and \eqref{eq:linear-small}, we have
	\begin{equation*}
		|F(x,t)-G(x)|\leq
		C\epsilon\Big(\frac{1}{H(S)}+\rho H(S)+\epsilon\rho H(S)^2\Big)
		\longrightarrow0,\quad\text{as }(x,t)\to(x_0,t_0),
	\end{equation*}
	where $\rho=S^{1/4}$ and $\epsilon>0$ is fixed. The perturbation is smooth at every point other than $(x_0,t_0)$ and is supported in a region where $|\nabla u|>0$.\ Outside its support, $F(x,t)=G(x)$.\ Hence $F\in C(\overline{B}_{1}\times[-1,0])$.\ At the distinguished spatial point there is even an exact local formula:
	\begin{equation*}
		F(x_0,t)=G(x_{0})-\frac{\epsilon}{H((t-t_0)^2)},
		\quad\text{for }0<|t-t_0|\text{ sufficiently small}.
	\end{equation*}
In particular, the source is genuinely time dependent.
	
	The weak equation has no hidden measure at $(x_0,t_0)$. One direct justification is to replace $H(S)$ by $H(S+\delta)$ for $0<\delta<1$ and let $\delta\rightarrow0$.\ The resulting perturbations and their spatial gradients converge uniformly.\ Their time derivatives and spatial second derivatives converge in every finite space--time $L^r$, by \eqref{eq:bounds} and dominated convergence.\ All derivative bounds are dominated by $C(1+H(S))$, whose every finite power is locally integrable.\ The flux differences are supported where $\nabla V$ stays away from zero, so passage in their weak formulations is legitimate.\ Thus $u\in C([-1,0];L^2(B_1))\cap L^p((-1,0);W^{1,p}(B_1))$ is an energy weak solution.
	
\noindent{\bf Step 3.} For $t\neq t_0$, it follows from monotonicity of $H$ and $S\geq(t-t_0)^2$ that
	\begin{equation}\label{eq:time-bound}
		\|\partial_tu(\cdot,t)\|_{L^\infty(B_1)}\leq C\big(1+H((t-t_0)^2)\big).
	\end{equation}
	Every finite power of the right-hand side is integrable in $t$.\ Integrating the time derivative, we obtain that for every $1<r<\infty$,
	\begin{align*}
		\|u(\cdot,t)-u(\cdot,s)\|_{L^\infty(B_1)}
		\leq \|\partial_tu\|_{L^r((-1,0);L^\infty(B_1))}
		|t-s|^{1-1/r}.
	\end{align*}
	The identity across $t_0$ follows by continuity and the integrable bound
	\eqref{eq:time-bound}. This proves every claimed H\"older exponent.
	On the other hand, we have
	\begin{equation*}
		\frac{|u(x_0,t_0+2h)-u(x_0,t_0+h)|}{h}=\epsilon\bigg(2\sqrt{\log\frac{e^4}{4h^2}}-\sqrt{\log\frac{e^4}{h^2}}\bigg)
		\longrightarrow\infty,\quad\text{as }h\rightarrow0.
	\end{equation*}
	Indeed, if $A_h=\log(e^4/h^2)$, the bracket is
	$\sqrt{A_h}-2\log4/(\sqrt{A_h-\log4}+\sqrt{A_h})$.
	Spatial continuity makes the $L^\infty$ norm at least the displayed
	interior point value. Both observation times lie strictly to the right
	of $t_0$, proving failure on every right neighborhood.
	
\noindent{\bf Step 4.} For the Dirichlet construction, we choose
	\begin{align*}
		V_{\scriptscriptstyle D}(x)=2+x_1,\quad x_0=0,\quad U=B_{\frac{1}{8}},
		\quad G=0,\quad \operatorname{tr}\mathcal{A}=n+p-2.
	\end{align*}
	Then \eqref{eq:u} reads explicitly
	\begin{equation*}
		u_{\scriptscriptstyle D}(x,t)=2+x_1+
		\epsilon\chi(x)\eta(t)
		\left(t+\frac{1}{2}+\frac{|x|^2}{2(n+p-2)}\right)
		\sqrt{\log\frac{e^4}{|x|^4+(t+\frac{1}{2})^2}},
	\end{equation*}
	with the perturbation defined to be zero at $(0,-\frac{1}{2})$.\ It vanishes near the lateral boundary and near $t=-1$. Thus both initial and boundary data are $2+x_1$. One can additionally ensure $\frac{1}{2}\leq|\nabla u_{\scriptscriptstyle D}|\leq\frac{3}{2}$ by shrinking the fixed $\epsilon$.
	
	For the Neumann construction, we pick 
	\begin{align*}
		V_{\scriptscriptstyle N}(x)=2+(1-|x|^{p'})^{p'},\quad
		x_0=\frac{1}{2}e_1,\quad U=B_{\frac{1}{8}}(x_0).
	\end{align*}
	This function is $C^1$ on the closed ball and smooth with nonzero gradient
	on $\overline U$. Its flux and source are explicitly
	\begin{align*}
		|\nabla V_{\scriptscriptstyle N}|^{p-2}\nabla V_{\scriptscriptstyle N}
		&=-(p')^{2(p-1)}(1-|x|^{p'})x,\quad G(x)=(p')^{2(p-1)}\big(n-(n+p')|x|^{p'}\big).
	\end{align*}
	The first expression vanishes on $\partial B_1$ and the second is continuous.\ Choose $\chi_{\scriptscriptstyle N}\in C_{0}^\infty(B_{1/8}(x_0))$,
	$0\leq\chi_{\scriptscriptstyle N}\leq1$, equal to one near $x_0$,
	and use the time cutoff $\eta$ from Step 1. The Neumann solution is explicitly
	\begin{equation*}
		\begin{aligned}
			&u_{\scriptscriptstyle N}(x,t)
			=2+(1-|x|^{p'})^{p'}\\
			&+\epsilon\chi_{\scriptscriptstyle N}(x)\eta(t)
			\left(
			t+\frac{1}{2}
			+\frac{|x-\frac{1}{2}e_1|^2}
			{2(n+p-2)|\nabla V_{\scriptscriptstyle N}(x_0)|^{p-2}}
			\right)\sqrt{\log\frac{e^4}
				{|x-\frac{1}{2}e_1|^4+(t+\frac{1}{2})^2}},
		\end{aligned}
	\end{equation*}
	where $|\nabla V_{\scriptscriptstyle N}(x_0)|=(p')^2 2^{1-p'}(1-2^{-p'})^{p'-1}$. The perturbation is defined to be zero at $(\frac{1}{2}e_1,-\frac{1}{2})$ and vanishes near $\partial B_1$ and $t=-1$. Hence $u_{\scriptscriptstyle N}(x,-1)=2+(1-|x|^{p'})^{p'}$ for $x\in B_1$. Near $\partial B_1$ the solution equals $V_{\scriptscriptstyle N}$, so its normal flux is zero. 
\end{proof}

\section{H\"{o}lder estimate in the supercritical singular case}\label{SEC006}
For $\theta_{\mathrm{c}}<p<2$, write
\begin{align*}
	\bar{\varepsilon}_{0}:=\frac{\varepsilon_{0}(2-p)}{p},
\end{align*}
where $\varepsilon_{0}$ is specified in \eqref{EXP01}. Since
\begin{align*}
2pl_{0}+n(p-2)-2(p(l_{0}-1)-n)=p(n+2)>0,\quad l_{0}>\frac{n+p}{p},	
\end{align*}	
we have
\begin{align}\label{E20}
0<\bar{\varepsilon}_{0}<\frac{2-p}{2}<\frac{1}{2}.	
\end{align}	
Note that for any $R\in(0,\frac{1}{4}]$, $x_{0}\in \overline{B}_{\frac{1}{2}}$, and $t_{0}\in[-\frac{1}{2},0]$,
\begin{align*}
	[(x_{0},t_{0})+Q(R^{1-\bar{\varepsilon}_{0}},R^{p})]\subset B_{1}\times\Big(-\frac{3}{4},0\Big]\subset Q(1,1).
\end{align*}
Due to the fact that $R\in(0,\frac{1}{4}]$ and $x_{0}\in \overline{B}_{\frac{1}{2}}$, it follows from \eqref{E20} that
\begin{align*}
	2^{-1}+R^{1-\bar{\varepsilon}_{0}}\leq2^{-1}+4^{-(1-\bar{\varepsilon}_{0})}\leq 1.
\end{align*}
This guarantees that the cylinder considered here lies inside $Q(1,1)$. Performing a suitable translation, we may assume that $(x_{0},t_{0})=(0,0)$. Set
\begin{align*}
	\mu^{+}=\sup\limits_{Q(R^{1-\bar{\varepsilon}_{0}},R^{p})}u,\quad\mu^{-}=\inf\limits_{Q(R^{1-\bar{\varepsilon}_{0}},R^{p})}u,
\end{align*}
and
\begin{align*}
	\omega=\mathop{osc}\limits_{Q(R^{1-\bar{\varepsilon}_{0}},R^{p})}u=\mu^{+}-\mu^{-}.
\end{align*}
Define the following cylinder:
\begin{align}\label{Z999}
	Q(a_{0}R,R^{p}),\quad a_{0}:=\Big(\frac{\sigma\omega}{A}\Big)^{\frac{p-2}{p}},
\end{align}
where $\sigma\in(0,1)$ and $A>1$ are constants to be determined later in terms of the data. The scaling exponent $\frac{p-2}{p}$ was previously identified in \cite{CD1988,D1993}.\ Denote
\begin{align}\label{Z20}
\mathcal{M}_{\ast}=
\mathcal{Q}_{4}+\|\phi\|_{L^{\infty}(B_{1})}+\overline{C}\|\psi\|_{L^{\frac{pl_{0}}{p-1}}(Q(1,1))}^{\frac{n+p}{(p-1)(n+2)}},
\end{align}	
where $\overline{C}$ is defined in Lemma \ref{LEMZ01} and $\mathcal{Q}_{4}$ is given by \eqref{Z21}. Moreover, we have
\begin{align}\label{Z23}
\mathcal{M}_{\ast}\geq2\|u\|_{L^{\infty}(Q(1,1))}+1.	
\end{align}	

Clearly, either $\omega\leq \sigma^{-1}AR^{\varepsilon_{0}}$ or $\omega>\sigma^{-1}AR^{\varepsilon_{0}}$. In the latter case, we have
\begin{align*}
	Q(a_{0}R,R^{p})\subset Q(R^{1-\bar{\varepsilon}_{0}},R^{p}),\quad \mathop{osc}\limits_{Q(a_{0}R,R^{p})}u\leq\omega.
\end{align*}
Denote
\begin{align}\label{M01}
	m:=\Big(\frac{\sigma\omega}{\mathcal{M}_{\ast}}\Big)^{\frac{p-2}{p}}\in [1,\infty).
\end{align}
If $A\geq\mathcal{M}_{\ast}$, we know $Q(mR,R^{p})\subset Q(a_{0}R,R^{p})$. 

Define
\begin{align}\label{AH26}
	\kappa_{0}:=-\frac{\ln(1-\sqrt{15}/4)}{\ln 2}>0.	
\end{align}	
We next establish the three key ingredients required for the parabolic De Giorgi truncation method.
\begin{lemma}\label{lem006}
With $\sigma$, $m$, and $\mathcal{M}_{\ast}$ as in \eqref{Z999}--\eqref{M01}, there exists a constant $\sigma_{0}\in(0,1]$, depending solely on the prescribed data and independent of $\sigma$, such that for all $0<\sigma<\sigma_{0}$,
	
	$(i)$ if
	\begin{align}\label{WMDN001}
		\frac{|B_{mR}\cap\{u(\cdot,-R^{p})>\mu^{+}-2^{-1}\omega\}|}{|B_{mR}|}\leq\frac{1}{2},
	\end{align}
	then one has either $\omega\leq \sigma^{-1}\mathcal{M}_{\ast}R^{\varepsilon_{0}}$, or
	\begin{align}\label{DQAF001}
		\frac{|B_{mR}\cap\{u(\cdot,t)>\mu^{+}-2^{-(\kappa_{0}+1)}\omega\}|}{|B_{mR}|}\leq\frac{4}{5},\quad\text{for all }t\in[-R^{p},0];
	\end{align}
	
	$(ii)$ if
	\begin{align*}
		\frac{|B_{mR}\cap\{u(\cdot,-R^{p})<\mu^{-}+2^{-1}\omega\}|}{|B_{mR}|}\leq\frac{1}{2},
	\end{align*}
	then one has either $\omega\leq \sigma^{-1}\mathcal{M}_{\ast}R^{\varepsilon_{0}}$, or
	\begin{align}\label{WZPMQ001}
		\frac{|B_{mR}\cap\{u(\cdot,t)<\mu^{-}+2^{-(\kappa_{0}+1)}\omega\}|}{|B_{mR}|}\leq\frac{4}{5},\quad\text{for all }t\in[-R^{p},0],
	\end{align}
	where $\kappa_{0}$ is specified in \eqref{AH26}.
	
\end{lemma}

\begin{proof}
We only prove \eqref{DQAF001}, as \eqref{WZPMQ001} can be established similarly.\ For notational simplicity, denote $v_{+}:=(u-(\mu^{+}-2^{-1}\omega))_{+}$.\ Theorem \ref{thm10}, combined with \eqref{WMDN001}, implies that for any $t\in[-R^{p},0]$,
\begin{align}\label{QNE089}
		\int_{B_{mR}}v_{+}^{2}(x,t)dx=&\int_{B_{mR}}v_{+}^{2}(x,-R^{p})dx+\int_{B_{mR}}(v_{+}^{2}(x,t)-v_{+}^{2}(x,-R^{p}))dx\notag\\
		\leq&\frac{\omega^{2}}{8}|B_{mR}|+\omega\int_{B_{mR}}|v_{+}(x,t)-v_{+}(x,-R^{p})|dx\notag\\
		\leq& \Big(\frac{\omega^{2}}{8}+\omega\mathfrak{S}_{\ast}R^{p\alpha}\Big)|B_{mR}|,
\end{align}
where $\alpha$ is given in \eqref{eq:unified-exponents} and the constant $\mathfrak{S}_{\ast}$ is determined by Theorem \ref{thm10}. Then for $\sigma\in(0,\sigma_{0})$ with $\sigma_{0}:=\frac{\mathcal{M}_{\ast}}{16\mathfrak{S}_{\ast}}\in(0,1),$ we deduce that if $\omega>\sigma^{-1}\mathcal{M}_{\ast}R^{\varepsilon_{0}}$, there holds
\begin{align*}
\omega\mathfrak{S}_{\ast}R^{p\alpha}\leq \frac{\omega^{2}}{16}.
\end{align*}	
Inserting this into \eqref{QNE089}, we derive that for $t\in[-R^{p},0]$, 
\begin{align}\label{Z26}
\int_{B_{mR}}v_{+}^{2}(x,t)dx\leq \frac{3\omega^{2}}{16}|B_{mR}|.	
\end{align}	

On the other hand, we see that for any $t\in[-R^{p},0]$, 
	\begin{align*}
		&\int_{B_{mR}}v_{+}^{2}(x,t)dx\geq\frac{\omega^{2}}{4}(1-2^{-\kappa_{0}})^{2}|B_{mR}\cap\{u(\cdot,t)>\mu^{+}-2^{-(\kappa_{0}+1)}\omega\}|,
	\end{align*}
	where $\kappa_{0}$ is defined in \eqref{AH26}.\ Combining this with \eqref{Z26}, it follows that for every $t\in[-R^{p},0]$,
	\begin{align*}
		&\frac{|B_{mR}\cap\{u(\cdot,t)>\mu^{+}-2^{-(\kappa_{0}+1)}\omega\}|}{|B_{mR}|}\leq\frac{3}{4(1-2^{-\kappa_{0}})^{2}}\leq\frac{4}{5}.
	\end{align*}
This completes the proof.

\end{proof}

\begin{lemma}\label{lem005}
Suppose that $\sigma,m,\mathcal{M}_{\ast}$ are introduced by \eqref{Z999}--\eqref{M01}.\ Then there exists a constant $\overline{C}_{0}>1$, whose value depends only on the prescribed data and is independent of $\sigma$, such that for all $\gamma\in(0,1)$ and $\delta\in(0,\frac{1}{2}]$,
	
	$(i)$ if
	\begin{align}\label{D09}
		\frac{| B_{mR/2}\cap\{u(\cdot,t)>\mu^{+}-\delta\omega\}|}{|B_{mR/2}|}\leq1-\gamma,\quad\forall\;t\in[-R^{p},0],
	\end{align}
	then for every $j\geq1$, we derive either
	\begin{align*}
		\omega\leq&\sigma^{-1}\max\Big\{\mathcal{M}_{\ast},(\delta^{-1}2^{j})^{\frac{p^{2}l_{0}}{2pl_{0}+n(p-2)}}\Big\}R^{\varepsilon_{0}},
	\end{align*}
	or
	\begin{align}\label{DEC001}
		\frac{|Q(mR/2,(R/2)^{p})\cap\{u>\mu^{+}-\frac{\delta\omega}{2^{j}}\}|}{|Q(mR/2,(R/2)^{p})|}\leq\frac{\overline{C}_{0}}{\gamma j^{\frac{p-1}{p}}};
	\end{align}
	
	$(ii)$ if
	\begin{align*}
		\frac{|B_{mR/2}\cap\{u(\cdot,t)<\mu^{-}+\delta\omega\}|}{|B_{mR/2}|}\leq1-\gamma,\quad\forall\;t\in[-R^{p},0],
	\end{align*}
	then for every $j\geq1$, we derive either
	\begin{align*}
			\omega\leq&\sigma^{-1}\max\Big\{\mathcal{M}_{\ast},(\delta^{-1}2^{j})^{\frac{p^{2}l_{0}}{2pl_{0}+n(p-2)}}\Big\}R^{\varepsilon_{0}},
	\end{align*}
	or
	\begin{align}\label{D99}
		\frac{|Q(mR/2,(R/2)^{p})\cap\{u<\mu^{-}+\frac{\delta\omega}{2^{j}}\}|}{|Q(mR/2,(R/2)^{p})|}\leq\frac{\overline{C}_{0}}{\gamma j^{\frac{p-1}{p}}}.
	\end{align}

\end{lemma}

\begin{proof}
We present only the proof of \eqref{DEC001}, since the argument for \eqref{D99} is analogous. For $i\geq0$, write $k_{i}=\mu^{+}-\frac{\delta\omega}{2^{i}}$, and
	\begin{align*}
		A_{i}(t)=B_{mR/2}\cap\{u(\cdot,t)>k_{i}\},\quad A_{i}=Q(mR/2,(R/2)^{p})\cap\{u>k_{i}\}.
	\end{align*}
From \eqref{D09}, we derive
	\begin{align*}
		|B_{mR/2}\setminus A_{i}(t)|\geq\gamma|B_{mR/2}|\geq C(n)\gamma(mR)^{n}.
	\end{align*}
This, along with \eqref{pro001}, establishes that
	\begin{align*}
		|A_{i+1}(t)|\leq\frac{C2^{i}mR}{\gamma\delta\omega }\int_{A_{i}(t)\setminus A_{i+1}(t)}|\nabla u|dx.
	\end{align*}
An application of H\"older's inequality after integration over $[-(R/2)^{p},0]$ gives that
	\begin{align*}
		|A_{i+1}|\leq&\frac{C2^{i}mR}{\gamma\delta\omega }\int_{A_{i}\setminus A_{i+1}}|\nabla u|dxdt\notag\\
		\leq& \frac{C2^{i}mR}{\gamma\delta\omega }\bigg(\int_{A_{i}\setminus A_{i+1}}|\nabla u|^{p}dxdt\bigg)^{\frac{1}{p}}|A_{i}\setminus A_{i+1}|^{\frac{p-1}{p}}\notag\\
		\leq& \frac{C2^{i}mR}{\gamma\delta\omega }\bigg(\int_{Q(mR/2,(R/2)^{p})}|\nabla (u-k_{i})_{+}|^{p}dxdt\bigg)^{\frac{1}{p}}|A_{i}\setminus A_{i+1}|^{\frac{p-1}{p}}.
	\end{align*}
Fix a cutoff function $\zeta\in C_{0}^{\infty}(B_{mR})$ satisfying the following properties:\ $\zeta=1$ in $B_{mR/2}$, $0\leq\zeta\leq1$ and $|\nabla\zeta|\leq\frac{4}{mR}$ in $B_{mR}$.\ To simplify notation, we write $v_{i}=(u-k_{i})_{+}$. Then it follows from Theorem \ref{thm10} that
\begin{align}\label{E26}
&\int_{B_{mR}}\big|\big(v_{i}^{2}(x,-(R/2)^{p})-v_{i}^{2}(x,0)\big)\big|dx\notag\\
&\leq2|B_{mR}|\sup_{Q(mR,(R/2)^{p})}v_{i} \|v_{i}(\cdot,-(R/2)^{p})-v_{i}(\cdot,0)\|_{L^{\infty}(B_{mR})}\notag\\
&\leq C2^{-i}\delta\omega (mR)^{n}R^{p\alpha},
\end{align}	
where $\alpha$ is specified in \eqref{eq:unified-exponents}. Note that 
\begin{align*}
\frac{p^{2}l_{0}}{2pl_{0}+n(p-2)}\geq p-1,\quad \frac{(2-p)(pl_{0}-n)}{2pl_{0}+n(p-2)}\leq1,
\end{align*}
which, combined with Lemma \ref{lem003} and \eqref{E26}, leads to that if
	\begin{align*}
		\omega>&\sigma^{-1}\max\Big\{\mathcal{M}_{\ast},(\delta^{-1}2^{i})^{\frac{p^{2}l_{0}}{2pl_{0}+n(p-2)}}\Big\}R^{\varepsilon_{0}}\notag\\
		\geq&\max\Big\{\sigma^{-1}\mathcal{M}_{\ast},\big((\delta^{-1}2^{i})^{p^{2}l_{0}}\sigma^{(p-2)(pl_{0}-n)}\big)^{\frac{1}{2pl_{0}+n(p-2)}}\Big\}R^{\varepsilon_{0}},
	\end{align*}
	then
	\begin{align*}
		&\frac{L_{1}}{6}\int_{Q(mR/2,(R/2)^{p})}|\nabla (u-k_{i})_{+}|^{p}dxdt\leq\frac{L_{1}}{6}\int_{Q(mR,(R/2)^{p})}|\nabla(v_{i}\zeta)|^{p}dxdt\notag\\
		&\leq \int_{B_{mR}}\big(v_{i}^{2}(x,-(R/2)^{p})-v_{i}^{2}(x,0)\big)\zeta dx+C\int_{Q(mR,(R/2)^{p})\cap\{u>k_{i}\}}v_{i}^{p}|\nabla\zeta|^{p}dxdt\notag\\
		&\quad+C|Q(mR,(R/2)^{p})\cap\{u>k_{i}\}|^{\frac{l_{0}-1}{l_{0}}}\notag\\
		&\leq C\Big(\frac{\delta\omega}{2^{i}} (mR)^{n}R^{p\alpha}+\frac{(\delta\omega)^{p}}{2^{pi}}m^{n-p}R^{n}+m^{\frac{n(l_{0}-1)}{l_{0}}}R^{\frac{(n+p)(l_{0}-1)}{l_{0}}}\Big)\leq \frac{C(\delta\omega)^{p}R^{n}}{2^{pi}m^{p-n}}.
	\end{align*}
The preceding facts imply that
	\begin{align*}
		|A_{i+1}|\leq&\frac{C}{\gamma}|A_{i}\setminus A_{i+1}|^{\frac{p-1}{p}}|Q(mR/2,(R/2)^{p})|^{\frac{1}{p}}.
	\end{align*}
Hence, it follows that for every $j>i\geq0$, if
	\begin{align*}
		\omega>&\sigma^{-1}\max\Big\{\mathcal{M}_{\ast},(\delta^{-1}2^{j})^{\frac{p^{2}l_{0}}{2pl_{0}+n(p-2)}}\Big\}R^{\varepsilon_{0}},
	\end{align*}
then
	\begin{align*}
		j|A_{j}|^{\frac{p}{p-1}}\leq&\sum^{j-1}_{i=0}|A_{i+1}|^{\frac{p}{p-1}}\leq\frac{C}{\gamma^{\frac{p}{p-1}}}|Q(mR/2,(R/2)^{p})|^{\frac{p}{p-1}}.
	\end{align*}
This concludes that \eqref{DEC001} holds.

\end{proof}

\begin{lemma}\label{LEM0035}
With $\sigma,m,\mathcal{M}_{\ast}$ defined as in \eqref{Z999}--\eqref{M01}, there is a constant $\gamma_{0}\in(0,1)$ depending
solely on the prescribed data and independent of $\sigma$ such that

$(i)$ if
\begin{align}\label{ZWZ007}
|Q(mR,R^{p})\cap\{u>\mu^{+}-\sigma\omega\}|\leq\gamma_{0}|Q(mR,R^{p})|,
\end{align}
then either $\omega\leq \sigma^{-1}\mathcal{M}_{\ast}R^{\varepsilon_{0}}$, or
\begin{align}\label{DZ001}
u\leq \mu^{+}-\frac{\sigma\omega}{2},\quad\mathrm{in}\text{ $Q(mR/2,(R/2)^{p})$;}
\end{align}

$(ii)$ if
\begin{align*}
|Q(mR,R^{p})\cap\{u<\mu^{-}+\sigma\omega\}|\leq\gamma_{0}|Q(mR,R^{p})|,
\end{align*}
then either $\omega\leq \sigma^{-1}\mathcal{M}_{\ast}R^{\varepsilon_{0}}$, or
\begin{align}\label{ZWZ009}
u\geq \mu^{-}+\frac{\sigma\omega}{2},\quad\mathrm{in}\text{ $Q(mR/2,(R/2)^{p})$.}
\end{align}

\end{lemma}

\begin{proof}
We only provide the proof of \eqref{DZ001}, as the proof of \eqref{ZWZ009} is similar.\ For $i\geq0,$ write
\begin{align*}
	r_{i}=\frac{R}{2}+\frac{R}{2^{i+1}},\quad k_{i}=\mu^{+}-\sigma\omega+\frac{\sigma\omega}{2}(1-2^{-i}).
\end{align*}
Introduce a smooth cutoff function $\xi_{i}\in C^{\infty}(Q(mr_{i},r_{i}^{p}))$ with the following properties: $0\leq\xi_{i}\leq1$ in $Q(mr_{i},r_{i}^{p})$, 	$\xi_{i}=1$ in $Q(mr_{i+1},r_{i+1}^{p})$, $\xi_{i}=0$ on $\partial_{pa}Q(mr_{i},r_{i}^{p})$, and
\begin{align*}
		|\nabla\xi_{i}|\leq\frac{2^{i+3}}{mR},\quad|\partial_{t}\xi_{i}|\leq\frac{2}{r_{i}^{p}-r_{i+1}^{p}}\leq\frac{C2^{i}}{R^{p}}.
\end{align*}
Utilizing \eqref{M01}, we have
\begin{align}\label{Z90}
	(\sigma\omega)^{2}=\Big(\frac{\sigma\omega}{m}\Big)^{p}\mathcal{M}_{\ast}^{2-p}.
\end{align}
Noting that $1<p<2$ implies $\xi_{i}^{2}\leq\xi_{i}^{p}$ in $Q(mr_{i},r_{i}^{p})$, we infer from \eqref{ENE01} and \eqref{Z90} that if $\omega>\sigma^{-1}\mathcal{M}_{\ast}R^{\varepsilon_{0}}$, then
\begin{align*}
	&\sup\limits_{t\in(-r_{i}^{p},0)}\int_{B_{mr_{i}}}|(u-k_{i})_{+}\xi_{i}|^{2}dx+\frac{L_{1}}{6}\int_{Q(mr_{i},r_{i}^{p})}|\nabla((u-k_{i})_{+}\xi_{i})|^{p}dxdt\notag\\
	&\leq C\int_{Q(mr_{i},r_{i}^{p})}\big((u-k_{i})_{+}^{2}|\partial_{t}\xi_{i}|+(u-k_{i})_{+}^{p}|\nabla\xi_{i}|^{p}\big)+C|Q(mr_{i},r_{i}^{p})\cap \{u>k_{i}\}|^{\frac{l_{0}-1}{l_{0}}}\notag\\
	&\leq C\bigg(\frac{2^{i}(\sigma\omega)^{2}}{R^{p}}+\Big(\frac{2^{i}\sigma\omega}{mR}\Big)^{p}\bigg)|Q(mr_{i},r_{i}^{p})\cap \{u>k_{i}\}|+C|Q(mr_{i},r_{i}^{p})\cap \{u>k_{i}\}|^{\frac{l_{0}-1}{l_{0}}}\notag\\
	&\leq C\bigg(\frac{2^{pi}(\sigma\omega)^{2}m^{\frac{n}{l_{0}}}}{R^{\frac{p(l_{0}-1)-n}{l_{0}}}}+1\bigg)|Q(mr_{i},r_{i}^{p})\cap \{u>k_{i}\}|^{\frac{l_{0}-1}{l_{0}}}\notag\\
	&\leq\frac{C2^{pi}(\sigma\omega)^{2}m^{\frac{n}{l_{0}}}}{R^{\frac{p(l_{0}-1)-n}{l_{0}}}}|Q(mr_{i},r_{i}^{p})\cap \{u>k_{i}\}|^{\frac{l_{0}-1}{l_{0}}}.
\end{align*}
This, along with Lemma \ref{lem001}, yields that
\begin{align}\label{ZE01}
	&\Big(\frac{\sigma\omega}{2^{i+2}}\Big)^{p\tilde{\chi}}|Q(mr_{i+1},r_{i+1}^{p})\cap \{u>k_{i+1}\}|\notag\\
	&=(k_{i+1}-k_{i})^{p\tilde{\chi}}|Q(mr_{i+1},r_{i+1}^{p})\cap \{u>k_{i+1}\}|\leq\int_{Q(mr_{i},r_{i}^{p})}|(u-k_{i})_{+}\xi_{i}|^{p\tilde{\chi}}dxdt\notag\\
	&\leq C\bigg(\sup\limits_{t\in(-r_{i}^{p},0)}\int_{B_{mr_{i}}}|(u-k_{i})_{+}\xi_{i}|^{2}dx\bigg)^{\chi-1}\int_{Q(mr_{i},r_{i}^{p})}|\nabla((u-k_{i})_{+}\xi_{i})|^{p}dxdt\notag\\
	&\leq\bigg(\frac{C2^{pi}(\sigma\omega)^{2}m^{\frac{n}{l_{0}}}}{R^{\frac{p(l_{0}-1)-n}{l_{0}}}}|Q(mr_{i},r_{i}^{p})\cap \{u>k_{i}\}|^{1-\frac{1}{l_{0}}}\bigg)^{\chi},
\end{align}
where $\chi$ and $\tilde{\chi}$ are given by \eqref{chi}. Denote
\begin{align*}
	H_{i}=\frac{|Q(mr_{i},r_{i}^{p})\cap \{u>k_{i}\}|}{|Q(mR,R^{p})|},\quad\mathrm{for}\;i\geq0.
\end{align*}
In light of \eqref{ZE01}, we derive
\begin{align*}
	&H_{i+1}\leq\Big(C2^{pi(1+\frac{\tilde{\chi}}{\chi})}H_{i}^{1-\frac{1}{l_{0}}}\Big)^{\chi}\leq\prod\limits^{i}_{s=0}\Big[\Big(C2^{p(1+\frac{\tilde{\chi}}{\chi})(i-s)}\Big)^{\chi}\Big]^{\big(\frac{\chi(l_{0}-1)}{l_{0}}\big)^{s}}H_{0}^{\left(\frac{\chi(l_{0}-1)}{l_{0}}\right)^{i+1}}\notag\\
	&\leq C^{\chi\sum\limits^{i}_{s=0}\left(\frac{\chi(l_{0}-1)}{l_{0}}\right)^{s}}2^{p(\chi+\tilde{\chi})\sum\limits_{s=0}^{i}(i-s)\left(\frac{\chi(l_{0}-1)}{l_{0}}\right)^{s}}H_{0}^{\left(\frac{\chi(l_{0}-1)}{l_{0}}\right)^{i+1}}\leq(\overline{C}_{\ast}H_{0})^{\left(\frac{\chi(l_{0}-1)}{l_{0}}\right)^{i+1}},
\end{align*}
where the constant $\overline{C}_{\ast}$ does not depend on $\sigma$. Then picking $\gamma_{0}=(2\overline{C}_{\ast})^{-1}$ in \eqref{ZWZ007}, we obtain
\begin{align*}
	H_{i+1}\leq2^{-\left(\frac{\chi(l_{0}-1)}{l_{0}}\right)^{i+1}}\longrightarrow0,\quad\mathrm{as}\;i\rightarrow\infty.
\end{align*}
This completes the proof.

\end{proof}

The combination of Lemmas \ref{lem006}, \ref{lem005} and \ref{LEM0035} gives the following result.
\begin{corollary}\label{AMZW01}
The constants $\sigma$ and $A$, followed by $a_{0}$ and $m$, can be determined, together with a constant $\kappa_{\ast}>1$, all depending only on the prescribed data, such that

$(i)$ if
\begin{align}\label{W008}
\frac{|B_{mR/2}\cap\{u(\cdot,-(R/2)^{p})>\mu^{+}-2^{-1}\omega\}|}{|B_{mR/2}|}\leq\frac{1}{2},
\end{align}
then either $\omega\leq \sigma^{-1}AR^{\varepsilon_{0}}=2^{\kappa_{\ast}}AR^{\varepsilon_{0}}$, or
\begin{align}\label{W009}
\sup\limits_{Q(mR/4,(R/4)^{p})}u\leq \mu^{+}-\frac{\omega}{2^{\kappa_{\ast}+1}};
\end{align}

$(ii)$ if
\begin{align*}
\frac{|B_{mR/2}\cap\{u(\cdot,-(R/2)^{p})<\mu^{-}+2^{-1}\omega\}|}{|B_{mR/2}|}\leq\frac{1}{2},
\end{align*}
then either $\omega\leq \sigma^{-1}AR^{\varepsilon_{0}}=2^{\kappa_{\ast}}AR^{\varepsilon_{0}}$, or
\begin{align}\label{W010}
\inf\limits_{Q(mR/4,(R/4)^{p})}u\geq \mu^{-}+\frac{\omega}{2^{\kappa_{\ast}+1}}.
\end{align}

\end{corollary}
\begin{proof}
We shall only establish \eqref{W009}, since the proof of \eqref{W010} proceeds along similar lines.\ We first pick a sufficiently large constant $j_{0}>3$, depending only on the prescribed data, such that
\begin{align}\label{W012}
2^{-(\kappa_{0}+1+j_{0})}\leq\sigma_{0},\quad5\overline{C}_{0}j_{0}^{-\frac{p-1}{p}}\leq\gamma_{0},
\end{align}
where $\kappa_{0},\,\sigma_{0},\,\overline{C}_{0},\,\gamma_{0}$ are, respectively, given by \eqref{AH26}, Lemmas \ref{lem006}, \ref{lem005} and \ref{LEM0035}. Then, combining Lemma \ref{lem006} with \eqref{W008} and applying Lemma \ref{lem005} with $\delta=\frac{1}{2^{\kappa_{0}+1}}$,
$\sigma=\frac{1}{2^{\kappa_{0}+1+j_{0}}}$, $\gamma=\frac{1}{5}$, and $j=j_{0}$, we infer from \eqref{W012} that either $\omega\leq \sigma^{-1}AR^{\varepsilon_{0}}=2^{\kappa_{0}+1+j_{0}}AR^{\varepsilon_{0}}$ with $A=\max\big\{\mathcal{M}_{\ast},2^{\frac{p^{2}l_{0}(j_{0}+\kappa_{0}+1)}{2pl_{0}+n(p-2)}}\big\}$,
or
\begin{align}\label{W013}
&\frac{|Q(mR/2,(R/2)^{p})\cap\{u>\mu^{+}-\frac{\omega}{2^{\kappa_{0}+1+j_{0}}}\}|}{|Q(mR/2,(R/2)^{p})|}\leq5\overline{C}_{0}j_{0}^{-\frac{p-1}{p}}\leq\gamma_{0}.
\end{align}
In view of \eqref{W013}, an application of Lemma \ref{LEM0035} with $\sigma=\frac{1}{2^{\kappa_{0}+1+j_{0}}}$ shows either $\omega\leq \sigma^{-1}AR^{\varepsilon_{0}}=2^{\kappa_{0}+1+j_{0}}AR^{\varepsilon_{0}}$, or
\begin{align*}
u\leq \mu^{+}-\frac{\omega}{2^{\kappa_{0}+2+j_{0}}},\quad\mathrm{in}\text{ $Q(mR/4,(R/4)^{p})$.}
\end{align*}
The proof is thus completed by taking $\kappa_{\ast}:=\kappa_{0}+1+j_{0}$.

\end{proof}

We are now ready to prove Theorem \ref{ZWTHM90} in the supercritical singular case.
\begin{proof}[Proof of Theorem \ref{ZWTHM90} in the case when $\theta_{\mathrm{c}}<p<2$]
Using Corollary \ref{AMZW01} and following the same argument in Section 5.1 of \cite{HLZ202601}, we establish Theorem \ref{ZWTHM90} in the case of $\theta_{\mathrm{c}}<p<2.$ 

\end{proof}

\section{H\"{o}lder estimate in the degenerate case}\label{SEC007}
 
Let
\begin{align*}
	\tilde{\varepsilon}_{0}:=(p-2)\varepsilon_{0},
\end{align*}
where $\varepsilon_{0}$ is defined by \eqref{EXP01}. In view of $p>2$, we derive that for any $R\in(0,\frac{1}{2}]$, $x_{0}\in\overline{B}_{\frac{1}{2}}$, and $t_{0}\in[-\frac{1}{2},0]$, 
\begin{align*}
	[(x_{0},t_{0})+Q(R,R^{p-\tilde{\varepsilon}_{0}})]\subset B_{1}\times(-1,0].
\end{align*}
Without loss of generality, we may assume that $(x_{0},t_{0})=(0,0)$ after a suitable translation. Denote
\begin{align*}
	\tilde{\mu}^{+}=\sup\limits_{Q(R,R^{p-\tilde{\varepsilon}_{0}})}u,\quad\tilde{\mu}^{-}=\inf\limits_{Q(R,R^{p-\tilde{\varepsilon}_{0}})}u,
\end{align*}
and
\begin{align*}
	\widetilde{\omega}=\mathop{osc}\limits_{Q(R,R^{p-\tilde{\varepsilon}_{0}})}u=\tilde{\mu}^{+}-\tilde{\mu}^{-}.
\end{align*}
Introduce the cylinder as follows:
\begin{align}\label{Z01}
	Q(R,\tilde{a}_{0}R^{p}),\quad \tilde{a}_{0}=\Big(\frac{\tilde{\sigma}\widetilde{\omega}}{\tilde{A}}\Big)^{2-p},
\end{align}
where $\tilde{\sigma}\in(0,1)$ and $\tilde{A}>1$ are constants depending only on the prescribed data and to be fixed later.\ We note that the exponent $2-p$ used in this scaling was already identified in \cite{D1986,D1993}.

We consider the two alternatives $\widetilde{\omega}\leq\tilde{\sigma}^{-1}\tilde{A}R^{\varepsilon_{0}}$ and $\widetilde{\omega}>\tilde{\sigma}^{-1}\tilde{A}R^{\varepsilon_{0}}$.\ Assuming the latter, we obtain
\begin{align*}
	Q(R,\tilde{a}_{0}R^{p})\subset Q(R,R^{p-\tilde{\varepsilon}_{0}}),\quad\mathop{osc}\limits_{Q(R,\tilde{a}_{0}R^{p})}u\leq\widetilde{\omega}.
\end{align*}
Define
\begin{align}\label{Z02}
	\widetilde{m}=\Big(\frac{\tilde{\sigma}\widetilde{\omega}}{\mathcal{M}_{\ast}}\Big)^{2-p},
\end{align}
where $\mathcal{M}_{\ast}$ is defined by \eqref{Z20}. From \eqref{Z23}, we see that $\widetilde{m}\geq1$. If $\tilde{A}\geq\mathcal{M}_{\ast}$, then we have $Q(R,\widetilde{m}R^{p})\subset Q(R,\tilde{a}_{0}R^{p})$. 

\begin{lemma}\label{lemz01}
Under the definitions of $\mathcal{M}_{\ast}$ and $\tilde{\sigma},\,\widetilde{m}$ given in \eqref{Z20} and \eqref{Z01}--\eqref{Z02}, there exists a constant $\tilde{\sigma}_{0}\in(0,1]$, depending solely on the prescribed data and independent of $\tilde{\sigma}$, such that the following holds for every $\tilde{\sigma}\in(0,\tilde{\sigma}_{0})$:
	
	$(i)$ if
	\begin{align}\label{Z05}
		\frac{|B_{R}\cap\{u(\cdot,-\widetilde{m}R^{p})>\tilde{\mu}^{+}-2^{-1}\widetilde{\omega}\}|}{|B_{R}|}\leq\frac{1}{2},
	\end{align}
	then either $\widetilde{\omega}\leq \tilde{\sigma}^{-1}\mathcal{M}_{\ast}R^{\varepsilon_{0}}$, or
	\begin{align}\label{Z06}
		\frac{|B_{R}\cap\{u(\cdot,t)>\tilde{\mu}^{+}-2^{-(\kappa_{0}+1)}\widetilde{\omega}\}|}{|B_{R}|}\leq\frac{4}{5},\quad\text{for any }t\in[-\widetilde{m}R^{p},0];
	\end{align}
	
	$(ii)$ if
	\begin{align*}
		\frac{|B_{R}\cap\{u(\cdot,-\widetilde{m}R^{p})<\tilde{\mu}^{-}+2^{-1}\widetilde{\omega}\}|}{|B_{R}|}\leq\frac{1}{2},
	\end{align*}
	then either $\widetilde{\omega}\leq \tilde{\sigma}^{-1}\mathcal{M}_{\ast}R^{\varepsilon_{0}}$, or
	\begin{align}\label{Z07}
		\frac{|B_{R}\cap\{u(\cdot,t)<\tilde{\mu}^{-}+2^{-(\kappa_{0}+1)}\widetilde{\omega}\}|}{|B_{R}|}\leq\frac{4}{5},\quad\text{for any }t\in[-\widetilde{m}R^{p},0],
	\end{align}
	where $\kappa_{0}$ is defined in \eqref{AH26}.
	
\end{lemma}

\begin{proof}
It is enough to establish \eqref{Z06}, as the proof of \eqref{Z07} proceeds along similar lines. For notational convenience, write $v_{+}:=(u-(\tilde{\mu}^{+}-2^{-1}\widetilde{\omega}))_{+}$. Making use of Theorem \ref{thm10} and \eqref{Z05}, we derive that for any $t\in[-\widetilde{m}R^{p},0]$,
	\begin{align}\label{Z09}
		\int_{B_{R}}v_{+}^{2}(x,t)dx=&\int_{B_{R}}v_{+}^{2}(x,-\widetilde{m}R^{p})dx+\int_{B_{R}}(v_{+}^{2}(x,t)-v_{+}^{2}(x,-\widetilde{m}R^{p}))dx\notag\\
		\leq&\frac{\widetilde{\omega}^{2}}{8}|B_{R}|+\widetilde{\omega}\int_{B_{R}}|v_{+}(x,t)-v_{+}(x,-\widetilde{m}R^{p})|dx\notag\\
		\leq& \Big(\frac{\widetilde{\omega}^{2}}{8}+\widetilde{\omega}\mathfrak{G}_{\ast} \big(\widetilde{m}R^{p}\big)^{\alpha}\Big)|B_{R}|,
	\end{align}
	where $\alpha$ is specified in \eqref{eq:unified-exponents} and $\mathfrak{G}_{\ast}$ is the constant determined by Theorem \ref{thm10}. Write $\tilde{\sigma}_{0}:=\frac{\mathcal{M}_{\ast}}{16\mathfrak{G}_{\ast}}\in(0,1)$. Therefore, for any $\tilde{\sigma}\in(0,\tilde{\sigma}_{0})$, if $\omega>\tilde{\sigma}^{-1}\mathcal{M}_{\ast}R^{\varepsilon_{0}},$ then we have
\begin{align*}
\widetilde{\omega}\mathfrak{G}_{\ast} \big(\widetilde{m}R^{p}\big)^{\alpha}\leq\frac{\widetilde{\omega}^{2}}{16},	
\end{align*}
which, combined with \eqref{Z09}, leads to that
\begin{align*}
	\int_{B_{R}}v_{+}^{2}(x,t)dx\leq\frac{3\widetilde{\omega}^{2}}{16}|B_{R}|.
\end{align*}			
	
	On the other hand, it follows that for every $t\in[-\widetilde{m}R^{p},0]$,
	\begin{align*}
		&\int_{B_{R}}v_{+}^{2}(x,t)dx\geq\frac{\widetilde{\omega}^{2}}{4}(1-2^{-\kappa_{0}})^{2}|B_{R}\cap\{u(\cdot,t)>\tilde{\mu}^{+}-2^{-(\kappa_{0}+1)}\widetilde{\omega}\}|,
	\end{align*}
	where $\kappa_{0}$ is given by \eqref{AH26}. Then we deduce that for any $t\in[-\widetilde{m}R^{p},0]$, 
	\begin{align*}
		&\frac{|B_{R}\cap\{u(\cdot,t)>\tilde{\mu}^{+}-2^{-(\kappa_{0}+1)}\widetilde{\omega}\}|}{|B_{R}|}\leq\frac{3}{4(1-2^{-\kappa_{0}})^{2}}\leq\frac{4}{5}.
	\end{align*}
This concludes the proof.

\end{proof}

\begin{lemma}\label{lemz02}
With $\mathcal{M}_{\ast}$ as in \eqref{Z20} and $\tilde{\sigma},\,\widetilde{m}$ as in \eqref{Z01}--\eqref{Z02}, there exists some constant $\widetilde{C}_{0}>1$, depending only on the prescribed data and not on $\tilde{\sigma}$, such that for all $\widetilde{\gamma}\in(0,1)$ and $\tilde{\delta}\in(0,\frac{1}{2}]$,
	
	$(i)$ if
	\begin{align}\label{Z10}
		\frac{| B_{R/2}\cap\{u(\cdot,t)>\tilde{\mu}^{+}-\tilde{\delta}\widetilde{\omega}\}|}{|B_{R/2}|}\leq1-\widetilde{\gamma},\quad\forall\;t\in[-\widetilde{m}R^{p},0],
	\end{align}
then for each fixed $j\geq1$ and every $0<\tilde{\sigma}\leq2^{-j}\tilde{\delta}$, one has either $$\widetilde{\omega}\leq\tilde{\sigma}^{-1}(2^{j}\tilde{\sigma}\tilde{\delta}^{-1})^{\frac{pl_{0}}{p(l_{0}-1)+2}}\mathcal{M}_{\ast}R^{\varepsilon_{0}},$$
 or
	\begin{align}\label{Z11}
		\frac{| Q(R/2,\widetilde{m}(R/2)^{p})\cap\{u>\tilde{\mu}^{+}-\frac{\tilde{\delta}\widetilde{\omega}}{2^{j}}\}|}{|Q(R/2,\widetilde{m}(R/2)^{p})|}\leq\frac{\widetilde{C}_{0}}{\widetilde{\gamma} j^{\frac{p-1}{p}}};
	\end{align}
	
	$(ii)$ if
	\begin{align*}
		\frac{|B_{R/2}\cap\{u(\cdot,t)<\tilde{\mu}^{-}+\tilde{\delta}\widetilde{\omega}\}|}{|B_{R/2}|}\leq1-\widetilde{\gamma},\quad\forall\;t\in[-\widetilde{m}R^{p},0],
	\end{align*}
then for each fixed $j\geq1$ and every $0<\tilde{\sigma}\leq2^{-j}\tilde{\delta}$, one has either $$\widetilde{\omega}\leq\tilde{\sigma}^{-1}(2^{j}\tilde{\sigma}\tilde{\delta}^{-1})^{\frac{pl_{0}}{p(l_{0}-1)+2}}\mathcal{M}_{\ast}R^{\varepsilon_{0}},$$
or
	\begin{align}\label{Z13}
		\frac{| Q(R/2,\widetilde{m}(R/2)^{p})\cap\{u<\tilde{\mu}^{-}+\frac{\tilde{\delta}\widetilde{\omega}}{2^{j}}\}|}{|Q(R/2,\widetilde{m}(R/2)^{p})|}\leq\frac{\widetilde{C}_{0}}{\widetilde{\gamma} j^{\frac{p-1}{p}}}.
	\end{align}
\end{lemma}

\begin{proof}
It is enough to establish \eqref{Z11}, since \eqref{Z13} can be proved similarly. For each integer $i\geq0$, write $k_{i}=\tilde{\mu}^{+}-\frac{\tilde{\delta}\widetilde{\omega}}{2^{i}}$, and
	\begin{align*}
		A_{i}(t)=B_{R/2}\cap\{u(\cdot,t)>k_{i}\},\quad A_{i}=Q(R/2,\widetilde{m}(R/2)^{p})\cap\{u>k_{i}\}.
	\end{align*}
Utilizing \eqref{Z10}, it follows that
	\begin{align*}
		|B_{R/2}\setminus A_{i}(t)|\geq\widetilde{\gamma}|B_{R/2}|\geq C(n)\widetilde{\gamma} R^{n}.
	\end{align*}
This, combined with \eqref{pro001}, leads to that
	\begin{align*}
		|A_{i+1}(t)|\leq\frac{C2^{i}R}{\widetilde{\gamma}\tilde{\delta}\widetilde{\omega}}\int_{A_{i}(t)\setminus A_{i+1}(t)}|\nabla u|dx,
	\end{align*}
and thus,
	\begin{align*}
		|A_{i+1}|\leq\frac{C2^{i}R}{\widetilde{\gamma}\tilde{\delta}\widetilde{\omega}}\int_{A_{i}\setminus A_{i+1}}|\nabla u|dxdt.
	\end{align*}	
Using $p>2$ and applying H\"older's inequality, we obtain
	\begin{align*}
		|A_{i+1}|\leq&\frac{C2^{i}R}{\widetilde{\gamma}\tilde{\delta}\widetilde{\omega}}|A_{i}\setminus A_{i+1}|^{\frac{p-1}{p}}\bigg(\int_{Q(R/2,\widetilde{m}(R/2)^{p})}|\nabla (u-k_{i})_{+}|^{p}dxdt\bigg)^{\frac{1}{p}}.
	\end{align*}
Take a time-independent cutoff function $\zeta\in C_{0}^{\infty}(B_{R})$ satisfying the following conditions:
	\begin{align*}
		\zeta=1\;\mathrm{in}\;B_{R/2},\quad\mathrm{and}\;\,0\leq\zeta\leq1,\;|\nabla\zeta|\leq\frac{4}{R}\;\,\mathrm{in}\;B_{R}.
	\end{align*}
For any $0<\tilde{\sigma}\leq2^{-i}\tilde{\delta}$, we have from \eqref{ENE01} and \eqref{Z02} that if
	\begin{align*}
		\widetilde{\omega}>\tilde{\sigma}^{-1}(2^{i}\tilde{\sigma}\tilde{\delta}^{-1})^{\frac{pl_{0}}{p(l_{0}-1)+2}}\mathcal{M}_{\ast}R^{\varepsilon_{0}},
	\end{align*}
	then
	\begin{align*}
		&\int_{Q(R/2,\widetilde{m}(R/2)^{p})}|\nabla (u-k_{i})_{+}|^{p}dxdt\leq\int_{Q(R,\widetilde{m}(R/2)^{p})}|\nabla ((u-k_{i})_{+}\zeta)|^{p}dxdt\notag\\
		&\leq \frac{6}{L_{1}}\int_{B_{R}} (u-k_{i})^{2}_{+}(x,-\widetilde{m}(R/2)^{p})\zeta^{p}(x)dx+C\int_{Q(R,\widetilde{m}(R/2)^{p})}(u-k_{i})_{+}^{p}|\nabla\zeta|^{p}dxdt\notag\\
		&\quad+C|Q(R,\widetilde{m}(R/2)^{p})\cap\{(u-k_{i})_{+}>0\}|^{1-\frac{1}{l_{0}}}\\
		&\leq C\Big(\frac{(\tilde{\delta}\widetilde{\omega})^{2}R^{n}}{2^{2i}}+\frac{(\tilde{\delta}\widetilde{\omega})^{p}\widetilde{m}R^{n}}{2^{pi}}+\big(\widetilde{m}R^{n+p}\big)^{1-\frac{1}{l_{0}}}\Big)\leq \frac{C(\tilde{\delta}\widetilde{\omega})^{p}\widetilde{m}R^{n}}{2^{pi}}.
	\end{align*}
Collecting the above results, we obtain
	\begin{align*}
		|A_{i+1}|\leq&\frac{C}{\widetilde{\gamma}}|A_{i}\setminus A_{i+1}|^{\frac{p-1}{p}}|Q(R/2,\widetilde{m}(R/2)^{p})|^{\frac{1}{p}}.
	\end{align*}
Consequently, for any given $j>i\geq0$ and $0<\tilde{\sigma}\leq2^{-j}\tilde{\delta}$, if
 $$	\widetilde{\omega}>\tilde{\sigma}^{-1}(2^{j}\tilde{\sigma}\tilde{\delta}^{-1})^{\frac{pl_{0}}{p(l_{0}-1)+2}}\mathcal{M}_{\ast}R^{\varepsilon_{0}},$$
 then we deduce
	\begin{align*}
		j|A_{j}|^{\frac{p}{p-1}}\leq&\sum^{j-1}_{i=0}|A_{i+1}|^{\frac{p}{p-1}}\leq\frac{C}{\widetilde{\gamma}^{\frac{p}{p-1}}}|Q(R/2,\widetilde{m}(R/2)^{p})|^{\frac{p}{p-1}}.
	\end{align*}
This establishes \eqref{Z11}.

\end{proof}

\begin{lemma}\label{lemz03}
Let $\mathcal{M}_{\ast}$ and $\tilde{\sigma},\,\widetilde{m}$ be defined as in \eqref{Z20} and \eqref{Z01}--\eqref{Z02}, respectively. Then we can choose a constant $\widetilde{\gamma}_{0}\in(0,1)$, depending only on the prescribed data but not on $\tilde{\sigma}$, such that
	
	$(i)$ if
	\begin{align}\label{Z15}
		|Q(R,\widetilde{m}R^{p})\cap\{u>\tilde{\mu}^{+}-\tilde{\sigma}\widetilde{\omega}\}|\leq\widetilde{\gamma}_{0}|Q(R,\widetilde{m}R^{p})|,
	\end{align}
	then either $\widetilde{\omega}\leq \tilde{\sigma}^{-1}\mathcal{M}_{\ast}R^{\varepsilon_{0}}$, or
	\begin{align}\label{Z16}
		u\leq \tilde{\mu}^{+}-\frac{\tilde{\sigma}\widetilde{\omega}}{2},\quad\mathrm{in}\text{ $Q(R/2,\widetilde{m}(R/2)^{p})$;}
	\end{align}
	
	$(ii)$ if
	\begin{align*}
		|Q(R,\widetilde{m}R^{p})\cap\{u<\tilde{\mu}^{-}+\tilde{\sigma}\widetilde{\omega}\}|\leq\widetilde{\gamma}_{0}|Q(R,\widetilde{m}R^{p})|,
	\end{align*}
	then either $\widetilde{\omega}\leq \tilde{\sigma}^{-1}\mathcal{M}_{\ast}R^{\varepsilon_{0}}$, or
	\begin{align}\label{Z17}
		u\geq \tilde{\mu}^{-}+\frac{\tilde{\sigma}\widetilde{\omega}}{2},\quad\mathrm{in}\text{ $Q(R/2,\widetilde{m}(R/2)^{p})$.}
	\end{align}
	
\end{lemma}

\begin{proof}
We restrict ourselves to proving \eqref{Z16}, since the argument for \eqref{Z17} is analogous. For any integer $i\geq0$, denote
	\begin{align*}
		r_{i}=\frac{R}{2}+\frac{R}{2^{i+1}},\quad k_{i}=\tilde{\mu}^{+}-\tilde{\sigma}\widetilde{\omega}+\frac{\tilde{\sigma}\widetilde{\omega}}{2}(1-2^{-i}).
	\end{align*}
Fix a smooth cutoff function $\xi_{i}\in C^{\infty}(Q(r_{i},\widetilde{m}r_{i}^{p}))$ with the following properties:\ $0\leq\xi_{i}\leq1$ in $Q(r_{i},\widetilde{m}r_{i}^{p})$, $\xi_{i}=1$ in $Q(r_{i+1},\widetilde{m}r_{i+1}^{p})$, $\xi_{i}=0$ on $\partial_{pa}Q(r_{i},\widetilde{m}r_{i}^{p}),$ and
	\begin{align*}
		|\nabla\xi_{i}|\leq\frac{2^{i+3}}{R},\quad|\partial_{t}\xi_{i}|\leq\frac{2}{\widetilde{m}(r_{i}^{p}-r_{i+1}^{p})}\leq\frac{C2^{i}}{\widetilde{m}R^{p}},\quad \mathrm{in}\;Q(r_{i},\widetilde{m}r_{i}^{p}).
	\end{align*}
Since $p>2$, a consequence of Lemma \ref{lem003} together with \eqref{Z02} shows that if $\widetilde{\omega}>\tilde{\sigma}^{-1}\mathcal{M}_{\ast}R^{\varepsilon_{0}}$,
	\begin{align}\label{Z18}
		&\sup\limits_{t\in(-\widetilde{m}r_{i}^{p},0)}(\tilde{\sigma}\widetilde{\omega})^{2-p}\int_{B_{r_{i}}}(u-k_{i})_{+}^{p}\xi_{i}^{p}dx+\frac{L_{1}}{6}\int_{Q(r_{i},\widetilde{m}r_{i}^{p})}|\nabla((u-k_{i})_{+}\xi_{i})|^{p}dxdt\notag\\
		&\leq\sup\limits_{t\in(-\widetilde{m}r_{i}^{p},0)}\int_{B_{r_{i}}}(u-k_{i})_{+}^{2}\xi_{i}^{p}dx+\frac{L_{1}}{6}\int_{Q(r_{i},\widetilde{m}r_{i}^{p})}|\nabla((u-k_{i})_{+}\xi_{i})|^{p}dxdt\notag\\
		&\leq C\int_{Q(r_{i},\widetilde{m}r_{i}^{p})}\big((u-k_{i})_{+}^{2}|\partial_{t}\xi_{i}|+(u-k_{i})_{+}^{p}|\nabla\xi_{i}|^{p}\big)dxdt\notag\\
		&\quad+C|Q(r_{i},\widetilde{m}r_{i}^{p})\cap\{u>k_{i}\}|^{1-\frac{1}{l_{0}}}\notag\\
		&\leq C\Big(\frac{2^{pi}(\tilde{\sigma}\widetilde{\omega})^{p}}{R^{p}}(\widetilde{m}R^{n+p})^{\frac{1}{l_{0}}}+1\Big)|Q(r_{i},\widetilde{m}r_{i}^{p})\cap\{u>k_{i}\}|^{1-\frac{1}{l_{0}}}\notag\\
		&\leq \frac{C2^{pi}(\tilde{\sigma}\widetilde{\omega})^{\frac{p(l_{0}-1)+2}{l_{0}}}}{R^{\frac{p(l_{0}-1)-n}{l_{0}}}}|Q(r_{i},\widetilde{m}r_{i}^{p})\cap\{u>k_{i}\}|^{1-\frac{1}{l_{0}}}.
	\end{align}
Write $\hat{u}(x,\hat{t})=u(x,\widetilde{m}\hat{t})$, $\hat{\xi}_{i}(x,\hat{t})=\xi_{i}(x,\widetilde{m}\hat{t})$, and
	\begin{align*}
		\hat{A}_{i}(\hat{t})=B_{r_{i}}\cap\{\hat{u}(\cdot,\hat{t})>k_{i}\},\quad\hat{A}_{i}=Q(r_{i},r_{i}^{p})\cap\{\hat{u}>k_{i}\}.
	\end{align*}
As a consequence of Lemma \ref{lem001} and \eqref{Z18}, we derive
	\begin{align*}
		\frac{(\tilde{\sigma}\widetilde{\omega})^{p}}{2^{p(i+2)}}|\hat{A}_{i+1}|^{\frac{1}{\chi}}=&(k_{i+1}-k_{i})^{p}|\hat{A}_{i+1}|^{\frac{1}{\chi}}\leq\|(\hat{u}_{i}-k_{i})_{+}\hat{\xi}_{i}\|^{p}_{L^{p\chi}(Q(r_{i},r_{i}^{p}))}\notag\\
		\leq&\sup\limits_{\hat{t}\in(-r_{i}^{p},0)}\int_{B_{r_{i}}}(\hat{u}-k_{i})_{+}^{p}\hat{\xi}_{i}^{p}dx+\int_{Q(r_{i},r_{i}^{p})}|\nabla((\hat{u}-k_{i})_{+}\hat{\xi}_{i})|^{p}dxd\hat{t}\notag\\
		\leq& \frac{C2^{pi}(\tilde{\sigma}\widetilde{\omega})^{p}}{R^{\frac{p(l_{0}-1)-n}{l_{0}}}}|Q(r_{i},r_{i}^{p})\cap\{\hat{u}>k_{i}\}|^{1-\frac{1}{l_{0}}},
	\end{align*}
	where $\chi=\frac{n+p}{n}$. Then if $\widetilde{\omega}>\tilde{\sigma}^{-1}\mathcal{M}_{\ast}R^{\varepsilon_{0}}$, we have
	\begin{align*}
		|\hat{A}_{i+1}|\leq&\bigg(\frac{C4^{pi}}{R^{\frac{p(l_{0}-1)-n}{l_{0}}}}|\hat{A}_{i}|^{1-\frac{1}{l_{0}}}\bigg)^{\chi}.
	\end{align*}
	Write
	\begin{align*}
		H_{i}=\frac{|\hat{A}_{i}|}{|Q(R,R^{p})|},\quad\text{for } i\geq0.
	\end{align*}
Hence we derive
\begin{align*}
	H_{i+1}\leq&(C4^{pi})^{\chi} H_{i}^{\frac{\chi(l_{0}-1)}{l_{0}}}\leq\prod\limits^{i}_{s=0}\big[(C4^{p(i-s)})^{\chi}\big]^{\big(\frac{\chi(l_{0}-1)}{l_{0}}\big)^{s}}H_{0}^{\left(\frac{\chi(l_{0}-1)}{l_{0}}\right)^{i+1}}\notag\\
	\leq&(\widetilde{C}_{\ast}H_{0})^{\left(\frac{\chi(l_{0}-1)}{l_{0}}\right)^{i+1}},
\end{align*}
where $\widetilde{C}_{\ast}>0$ depends solely on the prescribed data and not on $\tilde{\sigma}$. Applying \eqref{Z15} with $\widetilde{\gamma}_{0} =(2\widetilde{C}_{\ast})^{-1}$, we obtain
\begin{align*}
	H_{i+1}\leq2^{-\left(\frac{\chi(l_{0}-1)}{l_{0}}\right)^{i+1}}\longrightarrow0,\quad\mathrm{as}\;i\rightarrow\infty.
\end{align*}
The desired conclusion now follows.

\end{proof}

Following the proof of Corollary \ref{AMZW01} with minor modifications and using Lemmas \ref{lemz01}, \ref{lemz02} and \ref{lemz03}, we directly obtain the following result.
\begin{corollary}\label{coroz01}
The constants $\tilde{\sigma}$ and $\tilde{A}$ can be chosen first, followed by $\tilde{a}_{0}$ and $\widetilde{m}$, together with a constant $\tilde{\kappa}_{\ast}>1$, all depending solely on the prescribed data, such that
	
	$(i)$ if
	\begin{align*}
		\frac{|B_{R/2}\cap\{u(\cdot,-\widetilde{m}(R/2)^{p})>\tilde{\mu}^{+}-2^{-1}\widetilde{\omega}\}|}{|B_{R/2}|}\leq\frac{1}{2},
	\end{align*}
	then we obtain either $\widetilde{\omega}\leq \tilde{\sigma}^{-1}\tilde{A}R^{\varepsilon_{0}}=2^{\tilde{\kappa}_{\ast}}\tilde{A}R^{\varepsilon_{0}}$, or
	\begin{align*}
		\sup\limits_{Q(R/4,\widetilde{m}(R/4)^{p})}u\leq \tilde{\mu}^{+}-\frac{\widetilde{\omega}}{2^{\tilde{\kappa}_{\ast}+1}};
	\end{align*}
	
	$(ii)$ if
	\begin{align*}
		\frac{|B_{R/2}\cap\{u(\cdot,-\widetilde{m}(R/2)^{p})<\tilde{\mu}^{-}+2^{-1}\widetilde{\omega}\}|}{|B_{R/2}|}\leq\frac{1}{2},
	\end{align*}
	then we obtain either $\widetilde{\omega}\leq \tilde{\sigma}^{-1}\tilde{A}R^{\varepsilon_{0}}=2^{\tilde{\kappa}_{\ast}}\tilde{A}R^{\varepsilon_{0}}$, or
	\begin{align*}
		\inf\limits_{Q(R/4,\widetilde{m}(R/4)^{p})}u\geq \tilde{\mu}^{-}+\frac{\widetilde{\omega}}{2^{\tilde{\kappa}_{\ast}+1}}.
	\end{align*}
	
\end{corollary}

We are now in a position to prove Theorem \ref{ZWTHM90} in the degenerate case.
\begin{proof}[Proof of Theorem \ref{ZWTHM90} in the case when $p>2$]
With Corollary \ref{coroz01} at hand, we establish Theorem \ref{ZWTHM90} in the case of $p>2$ by using the same argument as in Section 4.1 of \cite{HLZ202602}.
\end{proof}

%

\noindent{\bf{\large Acknowledgements.}} Z. Zhao was partially supported by NSFC (No.\ 12501254).




\begin{thebibliography}{99}

\bibitem{AL1983} H.W. Alt and S. Luckhaus, Quasilinear elliptic-parabolic differential equations. Math. Z. 183 (1983), no. 3, 311-341.

\bibitem{ACM2004} F. Andreu-Vaillo, V. Caselles and J.M. Maz\'{o}n, Parabolic quasilinear equations minimizing linear growth functionals.
Progress in Mathematics, 223. Birkh\"{a}user Verlag, Basel, 2004.

\bibitem{ABR2001} S. Axler, P. Bourdon and W. Ramey, Harmonic function theory. Second edition. Graduate Texts in Mathematics, 137. Springer-Verlag, New York, 2001.

\bibitem{BV2004} G.I. Barenblatt and J.L. V\'{a}zquez, Nonlinear diffusion and image contour enhancement. Interfaces Free Bound. 6 (2004), no. 1, 31-54.

\bibitem{BDMS2018} V. B\"{o}gelein, F. Duzaar, P. Marcellini and C. Scheven, Doubly nonlinear equations of porous medium type. Arch. Ration. Mech. Anal. 229 (2018), no. 2, 503-545.

\bibitem{BDL2021} V. B\"{o}gelein, F. Duzaar and N. Liao, On the H\"{o}lder regularity of signed solutions to a doubly nonlinear equation. J. Funct. Anal. 281 (2021), no. 9, Paper No. 109173, 58 pp.

\bibitem{BDGLS202301} V. B\"{o}gelein, F. Duzaar, N. Liao and L. Sch\"{a}tzler, On the H\"{o}lder regularity of signed solutions to a doubly nonlinear equation. Part II. Rev. Mat. Iberoam. 39 (2023), no. 3, 1005-1037.

\bibitem{BDGLS202302} V. B\"{o}gelein, F. Duzaar, U. Gianazza, N. Liao and C. Scheven, H\"{o}lder Continuity of the Gradient of Solutions to Doubly Non-Linear Parabolic Equations. arXiv:2305.08539.

\bibitem{BDGLS2026} V. B\"{o}gelein, F. Duzaar, U. Gianazza, N. Liao and C. Scheven, Schauder estimates for parabolic $p$-Laplace systems. J. Lond. Math. Soc. (2) 113 (2026), no. 3, Paper No. e70463, 100 pp. 

\bibitem{BFV2018} M. Bonforte, A. Figalli and J.L. V\'{a}zquez, Sharp global estimates for local and nonlocal porous medium-type equations in bounded domains. Anal. PDE 11 (2018), no. 4, 945-982.

\bibitem{BSS2022} M. Bonforte, N. Simonov and D. Stan, The Cauchy problem for the fast $p$-Laplacian evolution equation. Characterization of the global Harnack principle and fine asymptotic behaviour. J. Math. Pures Appl. (9) 163 (2022), 83-131.

\bibitem{B2011} H. Brezis, Functional analysis, Sobolev spaces and partial differential equations. Universitext. Springer, New York, 2011.

\bibitem{CD1988} Y.Z. Chen and E. DiBenedetto, On the local behavior of solutions of singular parabolic equations. Arch. Rational Mech. Anal. 103 (1988), no. 4, 319-345.

\bibitem{C1991} H.J. Choe, H\"{o}lder regularity for the gradient of solutions of certain singular parabolic systems. Comm. Partial Differential Equations 16 (1991), no. 11, 1709-1732.

\bibitem{CM2011} A. Cianchi and V.G. Maz'ya, Global Lipschitz regularity for a class of quasilinear elliptic equations. Comm. Partial Differential Equations 36 (2011), no. 1, 100–133.

\bibitem{CM2014} A. Cianchi and V.G. Maz'ya, Global boundedness of the gradient for a class of nonlinear elliptic systems. Arch. Ration. Mech. Anal. 212 (2014), no. 1, 129–177.

\bibitem{DK2007} P. Daskalopoulos and C. Kenig, Degenerate diffusions. Initial value problems and local regularity theory. EMS Tracts in Mathematics, 1. European Mathematical Society (EMS), Z\"{u}rich, 2007.

\bibitem{D1957} E. De Giorgi, Sulla differenziabilit$\grave{\mathrm{a}}$ e l'analiticit$\grave{\mathrm{a}}$ delle estremali degli integrali multipli regolari. Mem. Accad. Sci. Torino. Cl. Sci. Fis. Mat. Nat. (3) 3 (1957), 25-43.

\bibitem{D1983} E. DiBenedetto, $C^{1+\alpha}$ local regularity of weak solutions of degenerate elliptic equations. Nonlinear Anal. 7 (1983), no. 8, 827-850.

\bibitem{DF1984} E. DiBenedetto and A. Friedman, Regularity of solutions of nonlinear degenerate parabolic systems. J. Reine Angew. Math. 349 (1984), 83-128.

\bibitem{DF198501} E. DiBenedetto and A. Friedman, H\"{o}lder estimates for nonlinear degenerate parabolic systems. J. Reine Angew. Math. 357 (1985), 1-22.

\bibitem{DF198502} E. DiBenedetto and A. Friedman, Addendum to: ``H\"{o}lder estimates for nonlinear degenerate parabolic systems''. J. Reine Angew. Math. 363 (1985), 217-220.

\bibitem{D1986} E. DiBenedetto, On the local behaviour of solutions of degenerate parabolic equations with measurable coefficients. Ann. Scuola Norm. Sup. Pisa Cl. Sci. (4) 13 (1986), no. 3, 487-535.

\bibitem{DKV1991} E. DiBenedetto, Y.C. Kwong and V. Vespri, Local space-analyticity of solutions of certain singular parabolic equations. Indiana Univ. Math. J. 40 (2) (1991), 741-765.

\bibitem{DK1992} E. DiBenedetto and Y.C. Kwong, Harnack estimates and extinction profile for weak solution of certain singular parabolic equations. Trans. Amer. Math. Soc. 330 (2) (1992), 783-811.

\bibitem{D1993} E. DiBenedetto, Degenerate parabolic equations. Universitext. Springer-Verlag, New York, 1993.

\bibitem{DGV2008} E. DiBenedetto, U. Gianazza and V. Vespri, Harnack estimates for quasi-linear degenerate parabolic differential equations. Acta Math. 200 (2008), no. 2, 181-209.

\bibitem{DGV200802} E. DiBenedetto, U. Gianazza and V. Vespri, Subpotential lower bounds for nonnegative solutions to certain quasi-linear degenerate parabolic equations. Duke Math. J. 143 (2008), no. 1, 1-15.

\bibitem{DGV2012} E. DiBenedetto, U. Gianazza and V. Vespri, Harnack's inequality for degenerate and singular parabolic equations. Springer Monographs in Mathematics. Springer, New York, 2012.

\bibitem{DLK2013} L. Diening, P. Lindqvist and B. Kawohl, Mini-Workshop:\ The $p$-Laplacian operator and applications. Oberwolfach Rep. 10 (2013), no. 1, 433-482.

\bibitem{DZ2022} H.J. Dong and H.Y. Zhu, Gradient estimates for singular parabolic $p$-Laplace type equations with measure data. Calc. Var. Partial Differential Equations 61 (2022), no. 3, Paper No. 86, 41 pp.

\bibitem{DZ2024} H.J. Dong and H.Y. Zhu, Gradient estimates for singular $p$-Laplace type equations with measure data. J. Eur. Math. Soc. (JEMS) 26 (2024), no. 10, 3939–3985.

\bibitem{DM2010} F. Duzaar and G. Mingione, Gradient estimates via linear and nonlinear potentials. J. Funct. Anal. 259 (2010), no. 11, 2961–2998. 

\bibitem{DM201002} F. Duzaar and G. Mingione, Local Lipschitz regularity for degenerate elliptic systems. Ann. Inst. H. Poincar\'{e} C Anal. Non Lin\'{e}aire 27 (2010), no. 6, 1361–1396. 

\bibitem{DM2011} F. Duzaar and G. Mingione, Gradient estimates via non-linear potentials. Amer. J. Math. 133 (2011), no. 4, 1093–1149.

\bibitem{E1982} L.C. Evans, A new proof of local $C^{1,\alpha}$ regularity for solutions of certain degenerate elliptic P.D.E. J. Differential Equations 45 (1982), no. 3, 356-373.

\bibitem{G1974} H.C. Greenwald, Lipschitz spaces on the surface of the unit sphere in Euclidean $n$-space. Pacific J. Math. 50 (1974), 63–80.

\bibitem{HLZ202601} X. Hao, Y. Li and Z.W. Zhao, Local behavior of solutions to inhomogeneous singular parabolic $p$-Laplace equations. arXiv:2609.18815.

\bibitem{HLZ202602} X. Hao, Y. Li and Z.W. Zhao, Regularity for solutions to inhomogeneous degenerate parabolic $p$-Laplace equations. arXiv:2610.01724.

\bibitem{I1989} A.V. Ivanov, H\"{o}lder estimates for quasilinear doubly degenerate parabolic equations. Zap. Nauchn. Sem. Leningrad. Otdel. Mat. Inst. Steklov. (LOMI) 171 (1989), Kraev. Zadachi Mat. Fiz. i Smezh. Voprosy Teor. Funktsi{\tiny{\u{l}}}. 20, 70-105, 185; translation in J. Soviet Math. 56 (1991), no. 2, 2320-2347.

\bibitem{I199401} A.V. Ivanov, H\"{o}lder estimates for equations of slow and normal diffusion type. Zap. Nauchn. Sem. S.-Peterburg. Otdel. Mat. Inst. Steklov. (POMI) 215 (1994), Differentsial'naya Geom. Gruppy Li i Mekh. 14, 130-136, 311; translation in J. Math. Sci. (New York) 85 (1997), no. 1, 1640-1644.

\bibitem{I199402} A.V. Ivanov, Regularity for doubly nonlinear parabolic equations. Zap. Nauchn. Sem. S.-Peterburg. Otdel. Mat. Inst. Steklov. (POMI) 209 (1994), Voprosy Kvant. Teor. Polya i Statist. Fiz. 12, 37-59, 261; reprinted in J. Math. Sci. 83 (1997), no. 1, 22-37.

\bibitem{I1995} A.V. Ivanov, H\"{o}lder estimates for a natural class of equations of fast diffusion type. Zap. Nauchn. Sem. S.-Peterburg. Otdel. Mat. Inst. Steklov. (POMI) 229 (1995), Chisl. Metody i Voprosy Organ. Vychisl. 11, 29-62, 322; translation in J. Math. Sci. (New York) 89 (1998), no. 6, 1607-1630.

\bibitem{JX2019} T.L. Jin and J.G. Xiong, Optimal boundary regularity for fast diffusion equations in bounded domains. Amer. J. Math. 145 (2023), no. 1, 151-219.

\bibitem{JRX2023} T.L. Jin, X. Ros-Oton and J.G. Xiong, Optimal regularity and fine asymptotics for the porous medium equation in bounded domains. J. Reine Angew. Math. 809 (2024), 269-300.

\bibitem{KSU2012} T. Kuusi, J. Siljander and J.M. Urbano, Local H\"{o}lder continuity for doubly nonlinear parabolic equations. Indiana Univ. Math. J. 61 (2012), no. 1, 399-430.

\bibitem{KL2000} J. Kinnunen and J.L. Lewis, Higher integrability for parabolic systems of $p$-Laplacian type. Duke Math. J. 102 (2000), no. 2, 253–271.

\bibitem{K1988} Y.C. Kwong, Interior and boundary regularity of solutions to a plasma type equation. Proc. Amer. Math. Soc. 104 (1988), no. 2, 472-478.

\bibitem{L1994} G.M. Lieberman, Gradient estimates for a new class of degenerate elliptic and parabolic equations. Ann. Scuola Norm. Sup. Pisa Cl. Sci. (4) 21 (1994), no. 4, 497-522.

\bibitem{L2006} G.M. Lieberman, H\"{o}lder regularity for the gradients of solutions of degenerate parabolic systems. Ukr. Mat. Visn. 3 (2006), no. 3, 352–373.

\bibitem{L2019} P. Lindqvist, Notes on the stationary $p$-Laplace equation. SpringerBriefs in Mathematics. Springer, Cham, 2019.

\bibitem{PS2008} Y. Peres and S. Sheffield, Tug-of-war with noise: a game-theoretic view of the $p$-Laplacian. Duke Math. J. 145 (2008), no. 1, 91-120.

\bibitem{PV1993} M.M. Porzio and V. Vespri, H\"{o}lder estimates for local solutions of some doubly nonlinear degenerate parabolic equations.
J. Differential Equations 103 (1993), no. 1, 146-178.


\bibitem{U2008} J.M. Urbano, The method of intrinsic scaling. A systematic approach to regularity for degenerate and singular PDEs. Lecture Notes in Mathematics, 1930. Springer-Verlag, Berlin, 2008.

\bibitem{V2007} J.L. V\'{a}zquez, The porous medium equation. Mathematical theory. Oxford Mathematical Monographs. The Clarendon Press, Oxford University Press, Oxford, 2007.

\bibitem{VV2022} V. Vespri and M. Vestberg, An extensive study of the regularity of solutions to doubly singular equations. Adv. Calc. Var. 15 (2022), no. 3, 435-473.

\end{thebibliography}
\end{document}